\documentclass[12pt,a4paper]{amsart}
\usepackage{a4,amssymb,latexsym}

\title[The QKP hierarchy]{The quaternionic KP hierarchy and conformally immersed 2-tori 
in the 4-sphere.}
\author{Ian McIntosh}
\address{Department of Mathematics \\ University of York \\ Heslington, York
YO10 5DD, UK}
\email{ian.mcintosh@york.ac.uk}
\date{Dec 02, 2010}

\newcommand{\qh}{quaternionic holomorphic\ }

\renewcommand{\H}{\mathbb{H}}
\newcommand{\C}{\mathbb{C}}

\newcommand{\R}{\mathbb{R}}
\renewcommand{\P}{\mathbb{P}}
\newcommand{\CP}{\mathbb{CP}}
\newcommand{\Ci}{\mathbb{C}_\infty}

\newcommand{\HP}{\mathbb{HP}}
\newcommand{\Z}{\mathbb{Z}}
\newcommand{\N}{\mathbb{N}}
\newcommand{\HH}{\underline{\mathbb{H}}}
\renewcommand{\forall}{\ \text{for all}\ }

\newcommand{\caA}{\mathcal{A}}
\newcommand{\caB}{\mathcal{B}}
\newcommand{\caD}{\mathcal{D}}
\newcommand{\caE}{\mathcal{E}}
\newcommand{\caG}{\mathcal{G}}

\newcommand{\caJ}{\mathcal{J}}
\newcommand{\caL}{\mathcal{L}}
\newcommand{\caM}{\mathcal{M}}
\newcommand{\caO}{\mathcal{O}}
\newcommand{\caS}{\mathcal{S}}
\newcommand{\caU}{\mathcal{U}}
\newcommand{\caV}{\mathcal{V}}
\newcommand{\caZ}{\mathcal{Z}}

\newcommand{\bt}{\mathbf{t}}

\newcommand{\fp}{\mathfrak{p}}
\newcommand{\fq}{\mathfrak{q}}

\newcommand{\dx}{\partial_x}
\newcommand{\dy}{\partial_y}
\newcommand{\ev}{\mathrm{ev}}
\newcommand{\sm}{\mathrm{sm}}
\newcommand{\QKP}{\mathrm{QKP}}
\newcommand{\KP}{\mathrm{KP}}
\newcommand{\Hom}{\operatorname{Hom}}
\newcommand{\End}{\operatorname{End}}
\newcommand{\Sp}{\operatorname{Sp}}
\newcommand{\Jac}{\operatorname{Jac}}
\newcommand{\Pic}{\operatorname{Pic}}

\newcommand{\Spec}{\operatorname{Spec}}
\newcommand{\alg}{\mathrm{alg}}

\newcommand{\res}{\mathrm{res}}

\newtheorem{thm}{Theorem}[section]
\newtheorem{prop}[thm]{Proposition}
\newtheorem{lem}[thm]{Lemma}
\newtheorem{cor}[thm]{Corollary}

\newtheorem{defn}[thm]{Definition}
\theoremstyle{remark}
\newtheorem{exam}[thm]{Example}
\newtheorem{rem}[thm]{Remark}

\numberwithin{equation}{section}

\begin{document}

\begin{abstract}
The quaternionic KP hierarchy is the integrable hierarchy of p.d.e obtained by replacing the
complex numbers with the quaternions in the standard construction of the KP hierarchy and 
its solutions; it is equivalent to
what is often called the Davey-Stewartson II hierarchy. This article studies its relationship with
the theory of conformally immersed tori in the 4-sphere via quaternionic holomorphic geometry.
The Sato-Segal-Wilson construction of KP solutions is adapted to this
setting and the connection with \qh curves is made. We then compare three different notions 
of ``spectral curve'': the QKP 
spectral curve; the Floquet multiplier spectral curve for the related Dirac operator;
and the curve parameterising Darboux transforms of a conformal 2-torus in the 4-sphere.

\end{abstract}

\maketitle

\section{Introduction.}
There is a well established connection between the Davey-Stewartson 
(DS) hierarchy and conformal immersions of tori in $\R^3$ and $\R^4$
(a good survey can be found in \cite{Taisurv}). One compelling reason for studying
this correspondence is that the simplest ``first
integral'' of this hierarchy represents the Willmore functional, and it may be
that a strategy based upon integrable systems could resolve
the Willmore conjecture (see \cite{Sch} and \cite{FerLPP} for two different
perspectives on this).

But the original point of view makes it difficult 
to see the conformal invariance and presents other problems in $\R^4$ (see \cite{Tai06}). 
These difficulties are avoided by using the theory of quaternionic holomorphic curves, developed 
by the Berlin school \cite{PedP,BurFLPP,FerLPP,Boh,BohLPP,Boh08}.
The conformal $4$-sphere is thought of as $\HP^1$;
a conformal immersion of a surface $M$ becomes a quaternionic line bundle over $M$ whose 
dual bundle possesses a canonical quaternionic holomorphic structure $D$, the 
quaternionic analogue of a $\bar\partial$-operator. 
This operator $D$ is essentially the Dirac operator which plays the central role 
in the DS hierarchy.  

The relationship between the DS hierarchy and 
the geometry of the immersed torus is still not very well understood. There is at present no 
direct definition of the flows as deformations of a torus
(although, see \cite{BurPP} for an attempt at this);
there is only an indirect construction for tori of finite type via the linear motion on the
spectral curve. For the spectral curve, Taimanov \cite{Tai98} proposed 
the normalisation of the analytic curve of Floquet multipliers of the Dirac operator:
generically this non-compact with ``infinite genus''. 
The quaternionic perspective \cite{BohLPP,Boh08,BohPP}
provides a conformally invariant construction, and 
also gives a geometric interpretation of the smooth points on this curve, as parameterising
``Darboux transformations'' of the immersed torus. 

This article shows how to use the quaternionic point of view to give a more natural 
link to the integrable hierarchy. The hierarchy can be
obtained by replacing the field $\C$ with the division algebra $\H$ in the standard
constructions of the KP hierarchy and its solutions (particularly, following
\cite{Wil79,Wil81,SegW}). This point of view  is sufficiently rewarding that I feel the 
new name ``quaternionic KP'' (QKP) hierarchy is justified.
The plan of this article is to start with the QKP theory and work towards \qh geometry, paying
particular attention to the spectral data.

Properly thought of, the QKP hierarchy is a collection of Lax equations $L_t=[P,L]$ for a
\emph{pseudo-differential} operator with quaternionic coefficients
\[
L = i\dy + U_0 + U_1\dy^{-1}+\cdots,
\]
and a differential operator $P$. These equations provide
a commuting family of derivations on a real differential
algebra. One of these provides a real derivation $\dx$ for which 
\[
\caD=\frac{1}{2}(\dx+i\dy + U_0),
\]
plays the role of the Dirac operator. 

This construction is purely algebraic: to reconnect with analysis, we adopt the 
Sato-Segal-Wilson \cite{SegW} point of view
and realise the QKP equations as flows on a manifold, the quotient of an infinite 
dimensional Grassmannian.  This Grassmannian
$Gr(\H)$ possesses pretty much all the properties that hold for KP, provided we are careful to
distinguish between the cases where the Dirac potential $U_0$ is trivial or non-trivial. 
In particular, we show that:
\begin{enumerate}
\item the QKP flows correspond to the action of an infinite dimensional abelian Lie group
$\Gamma_+$ on $Gr(\H)$ and the QKP solutions are parameterized by a quotient manifold $\caM$. 
This manifold is the disjoint union of two manifolds, $\caM_\KP$, which is a copy of the 
ordinary KP phase space and
parameterises QKP solutions for which $U_0$ is trivial, and
$\caM_\QKP$, which parameterizes solutions with non-trivial $U_0$.
\item the points of $\caM$ which have finite dimensional $\Gamma_+$-orbits correspond precisely
to the solutions which can be constructed from spectral data (i.e., a complete algebraic curve
$X$ and other algebraic data on it), and the orbits themselves are open
subvarieties of (generalised) Jacobi varieties. We call these ``of finite type''.
\end{enumerate}
The object which mediates between Lax equations and points of $Gr(\H)$ is the (quaternionic) Baker
function. This is an $\H$-valued solution to the spectral problem $L\psi=-\psi\zeta$ which admits 
a left $\H$ Fourier series expansion of the form
\[
\psi(x,y,\zeta) = (1+a_1(x,y)\zeta^{-1}+\cdots)\exp[(x+iy)\zeta],\quad \zeta\in S^1.
\]
Points of the Grassmannian $Gr(\H)$ parameterise Baker functions which converge on $|\zeta|>1$.  
For the QKP solutions of finite type there is a
representative Baker function which extends to a globally
meromorphic function on the spectral curve. We call this the global Baker function: it plays
an important role in translating between the different notions of spectral curve. 

The Baker function satisfies $\caD\psi=0$ for all $\zeta$ and 
can therefore be used to obtain homogeneous coordinates of a \qh curve $f:\C\to\HP^n$, by 
evaluating $\psi$ at different values of $\zeta$. When $\psi$ is global these points give 
a divisor $\fq$ on $X$ and we show that the (double) periodicity of such a curve is 
governed by the Jacobi variety of a singularisation $X_\fq$ of $X$. 
In that case $f$ projects to a \qh torus in $\HP^1$
under any homogeneous projection from $\HP^n$: we can think of $f$ as a 
common ``linear system'' for each of these.

In this passage between QKP and \qh curves two obstacles remain. The first is that 
the periodicity conditions on the QKP spectral data imply
that the whole QKP operator $L$ is ``periodic'' (the precise meaning of this is given in Theorem
\ref{thm:periodic}(c)).
However, given a conformal torus then it seems one only obtains 
``periodicity'' of the Dirac operator $\caD$. There is a Baker-type function in the kernel of the
Dirac operator when the multiplier spectral curve has finite type. But I was unable to show 
that the two Baker functions coincide without knowing \emph{a priori} that the QKP operator $L$ is
``periodic'' (Remark \ref{rem:fperiod}). Secondly, when the multiplier spectral curve is not of 
finite type it seems there can be no convergent
multiplier Baker function, which raises the problem of how to make the passage between QKP and
conformal tori at all. In particular, it is not clear how to describe integrable deformations in
this case. This may be possible using the theory of infinite genus Riemann surfaces
\cite{FelKT,Schbook},
but that would take us out of the Grassmannian class of solutions discussed here.

In the case where the two Baker functions are the same section 6 relates three spectral
curves: the QKP spectral curve, the multiplier spectral curve, and what we call
the Darboux spectral curve. The Darboux spectral curve has a particularly neat characterisation
via an Abel map image of a punctured version of the QKP spectral curve $X$: when $X$ is smooth the
punctured curve is just $X\setminus\fq$, whose Abel image lies in $\Jac(X_\fq)$, which
possesses a natural map to $\CP^{2n+1}$. 
The multiplier and Darboux spectral curves can both be obtained from $X_\fq$, which suggests
that the singular curve $X_\fq$ is the correct object for 
classifying \qh tori of finite type. 

\textit{Acknowledgments.} 
I would like to thank Christoph Bohle, Katrin Leschke, Franz Pedit and Ulrich Pinkall for 
letting me see preliminary versions of \cite{BohLPP,BohPP} and for 
explaining their work on quaternionic holomorphic curves to me over a number of years.
In particular, the idea that there might be a ``quaternionic KP'' explanation for the
spectral curve of a conformally immersed torus is due to Ulrich Pinkall.
Some of these results were announced in the conference proceedings \cite{McI04}.

\section{Quaternionic holomorphic curves.}\label{sec:quat}

We begin by summarising the correspondence between
conformally immersed surfaces in $S^4\simeq\HP^1$ (as the space of left quaternionic lines in
$\H^2$) and quaternionic holomorphic 
curves (based on \cite{BurFLPP,FerLPP}).  Each immersed surface $f:M\to S^4$ can be equated with a
smooth (quaternionic) line subbundle $L$ of the trivial bundle 
$\HH^2=M\times\H^2$ via $f(p)=L_p$. 

The reader is warned at this point that we do not
adopt the convention of \cite{FerLPP} that all quaternionic bundles are right $\H$ bundles,
since this does not fit well with our construction of the QKP hierarchy using differential
operators, with coefficients over $\H$, acting on the left on $\H$-valued functions. 
Therefore $L$ will be a left $\H$-bundle and its quaternionic dual $L^*=\Hom_\H(L,\H)$ a 
right $\H$-bundle. 

The theory of \qh curves encodes the conformality of the map $f$ as follows.
The differential $df$ of any immersed surface $f:M\to\HP^1$ can be represented,
after the standard manner for maps into projective spaces, by 
$\delta f\in\Omega^1_M(\Hom(L,\HH^2/L))$ where
\[
\delta f:\psi\mapsto d\psi\bmod L,\quad \psi\in\Gamma(L).
\]
Here $\Gamma(L)$ denotes the space of smooth sections of $L$.
We also adopt the convention that, for any vector bundle valued 1-form $\omega$, $*\omega =
\omega\circ J_M$ where $J_M$ is the complex structure on $M$. Now $f$ is
a conformal immersion if and only if
there exists $J\in\End_\H(L)$ with $J^2=-I$ for which $*\delta f = \delta
f\circ J$. The pair $(L,J)$ is called a quaternionic holomorphic line bundle.

Further, $L^*$ inherits a complex structure, which we will also call $J$, and a 
canonical quaternionic holomorphic structure
\[
D:\Gamma(L^*)\to\Omega^1(L^*).
\]
This is an $\H$-linear operator satisfying, for any $\Psi\in \Gamma(L^*)$ and
$\mu:M\to \H$:
\[
*D\Psi  = -JD\Psi,\quad
D(\Psi\mu)  = (D\Psi)\mu +\frac{1}{2}(\Psi d\mu +J\Psi*d\mu).
\]
It is characterised by the property that the natural inclusion
$(\H^2)^*\subset \Gamma(L^*)$ maps into the kernel of $D$, i.e., into
the space $H^0_D(L^*)$ of quaternionic holomorphic sections (see \cite[\S 4.3]{BurFLPP}). 

The definitions above apply equally well to line subbundles of $\HH^{n+1}$ and
provide the notion of a \qh curve in $\HP^n$. 
Conversely, given a complex quaternionic line bundle $(E,J)$ equipped with a \qh
structure $D$, define $H^0_D(E)$ to be the kernel of $D$. When $M$ is
compact this space is finite dimensional (indeed, a Riemann-Roch formula can
be derived to calculate this dimension: see \cite{FerLPP}). 
The ``Kodaira construction'' follows through to give a \qh curve
\begin{equation}\label{eq:Kodaira}
f:M\to \P H^0_D(E)^*;\quad p\mapsto E_p^*.
\end{equation}
A two dimensional $\H$-subspace
$H\subset H^0_D(E)$ can be thought of as a linear system, in the sense of
algebraic curve theory. Provided $H$ is well-positioned (has no base points) the dual projection
$H^0_D(E)^*\to H^*$ gives us a conformal immersion $f:M\to \P H^*\simeq\HP^1$.

\subsection{The Dirac operator for a conformal torus in $S^4$ with flat normal
bundle.}\label{subsec:Dirac}
For any complex quaternionic line bundle $(L,J)$ we define $\hat L$ to be the complex line bundle
whose fibres are 
\[
\hat L_P = \{\sigma\in L_P;J\sigma = i\sigma \}.
\]
By definition, $\deg(L)$ is the degree of the complex line bundle $\hat L$. 
Bohle \cite[p.19]{Boh} showed that when $M$ is a compact Riemann surface
the degree of the normal bundle of $f$ is $2\deg(L)-\deg(K_M)$, where $K_M$ is the canonical bundle. 

Suppose we have a complex (right) quaternionic line bundle $(E,J)$ over a torus 
$M$, with \qh structure $D$. 
Let us assume this has degree zero, i.e., $\deg(\hat E)=0$.
The complex bundle $\hat E$ inherits the complex holomorphic structure
\[
\bar\partial_J = \frac{1}{2}(D-JDJ).
\]
The moduli space of degree zero holomorphic line bundles is isomorphic to the
moduli space $H^1(M,S^1)$ of Hermitian line bundles possessing a flat Hermitian connexion, so
$\bar\partial_J$ can be extended in a unique way to a flat connexion $\hat\nabla$ 
which is also Hermitian with respect to some Hermitian inner product on $\hat E$. 
Now, if we represent $M$ as $\C/\Lambda$ for some lattice
$\Lambda$ and let $\pi:\C\to M$ denote the universal cover then we can
trivialise $\pi^*\hat E$ 
with a smooth $\hat\nabla$-parallel section $\Phi$ which is therefore
$\bar\partial_J$-holomorphic. This section is unique up to right multiplication by a complex constant
and has unimodular monodromy $\mu\in\Hom(\Lambda,S^1)$, i.e., for each
$\lambda\in\Lambda$ and all $z\in\C$
\[
\Phi(z+\lambda) = \Phi(z)\mu(\lambda).
\]
For any $\Psi\in\Gamma(E)$ we therefore have some $\H$-valued function
$\psi$ on $\C$ for which $\Psi= \Phi\psi$ and we observe that $\psi$ has
monodromy $\mu^{-1}$ along the lattice $\Lambda$. It follows that
\[
D\Psi= (D\Phi)\psi +\Phi\frac{1}{2}(d\psi+i*d\psi).
\]

Now recall the decomposition $D=\bar\partial_J + Q$, where $Q=(D+JDJ)/2$ is called the ``Hopf
field'', and write $Q\Phi = \Phi d\bar z U $ for some $U:\C\to\H$.
Since $Q$ anti-commutes with $J$ it follows that $U$ anti-commutes with $i$.
Then  
\begin{equation}
\label{eq:Dirac}
D(\Phi\psi)= (\Phi d\bar z) \caD\psi,\qquad \caD = \partial/\partial\bar z + U. 
\end{equation}
We call $\caD$ the Dirac operator. For simplicity let us define
\[
\ker(\caD) = \{\psi\in C^\infty(\C,\H);\caD\psi =0\}.
\]
Then we have an identification 
\begin{equation}\label{eq:H^0_D}
H^0_D(E)\simeq\{\psi\in \ker(\caD);\psi(z+\lambda) = \mu^{-1}(\lambda)\psi(z) 
\forall\lambda\in\Lambda\}
\end{equation}
Notice that $\caD$ does not in general preserve the space of functions on 
the torus $M$, since
for any $\lambda\in\Lambda$ and all $z\in\C$
\begin{equation}\label{eq:monodromy}
U(z+\lambda) = \mu(\lambda)^{-1}U(z)\mu(\lambda) = U(z)\mu(\lambda)^2.
\end{equation}
Therefore $\caD$ is doubly periodic if and only if $\hat E$ is a spin bundle. We know from
\cite{PedP} that this is the case which corresponds to immersions into $\R^3$ (see also
\cite{Boh}). By \eqref{eq:monodromy} $|U|$ is always a function on $M$ itself and the 
$L^2$-norm of $U$ over $M$ is
called the Willmore energy of $\caD$: it is essentially the Willmore functional for the
corresponding conformally immersed torus.


\subsection{The multiplier spectrum and the Baker function.}
For any degree zero \qh (right $\H$)
line bundle $(E,D)$ over a torus $\C/\Lambda$ let $(\pi^*E,D)$ denote its pull-back
to the universal cover $\pi:\C\to\C/\Lambda$. Using this we construct the
multiplier spectrum of $(E,D)$:
\[
\Sp(E,D) = \{\chi\in\Hom(\Lambda,\C^\times); \exists\Psi\in H^0_D(\pi^*E),\ \Psi\neq 0,\ 
\lambda^*\Psi = \Psi\chi(\lambda)\forall \lambda\in\Lambda\}.
\]
By fixing two generators for $\Lambda$ we can identify $\Sp(E,D)$ with a subset of
$(\C^\times)^2$.
It has been shown (see \cite{Taisurv} for a survey of results) that $\Sp(E,D)$ is a complex
analytic curve. Right multiplication on $H^0_D(\pi^*E)$ by $j$ induces on
$\Sp(E,D)$ a real involution 
$\chi\mapsto\bar\chi$. Using the isomorphism \eqref{eq:H^0_D} we can write
\begin{multline}\label{eq:Sp}
\Sp(E,D) = \{\chi\in\Hom(\Lambda,\C^\times);\exists\psi\in\ker(\caD) ,\psi\neq 0,\\
\lambda^*\psi = \mu(\lambda)^{-1}\psi\chi(\lambda)\forall\lambda\in\Lambda\}.
\end{multline}
Here $\mu$ is the monodromy for the flat bundle $L^*$. 

The asymptotic structure of this spectrum is quite well understood (see \cite{BohLPP,Taisurv}) and is 
described by comparison with the spectrum of the ``vacuum'' operator $D_0=\bar\partial_J$, with
vacuum Dirac operator $\caD_0 = \partial/\partial\bar z$. Let us assume for the moment that 
$E$ is trivial (i.e.,
$\mu=1$). Taking $\psi(z,\zeta) = e^{z\zeta}$, for $\zeta \in\C$, we observe that $\caD_0\psi=0$ for
all $\zeta$ and 
\[
\psi(z+\lambda,\zeta) = \psi(z,\zeta)e^{\lambda\zeta},\quad\forall z\in\C,\ \lambda\in\Lambda.
\]
Hence $\Sp(E,D_0)$ contains an analytic curve of monodromies 
\[
C=\{(e^{\lambda_1\zeta},e^{\lambda_2\zeta})\in(\C^\times)^2;\zeta\in\C\},\quad \Lambda =
\Z\langle\lambda_1,\lambda_2\rangle.
\]
The full multiplier spectrum of $\caD_0$ is $C\cup\bar C$. These two branches of $\Sp(E,D_0)$ intersect
infinitely often in double points. When the monodromy $\mu$ is nontrivial the structure is the same
but with the branches $C$ and $\bar C$ shifted by appropriate factors.

The spectrum $\Sp(E,D)$ is asymptotic to $\Sp(E,D_0)$ in the sense that outside a compact subset of
$(\C^\times)^2$ the former is contained in an arbitrarily small tube around the latter. Away from
the double points of $\Sp(E,D_0)$ the curve $\Sp(E,D)$ is a graph over $\Sp(E,D_0)$, while near each  
double point $\Sp(E,D)$ either has a double point itself or is annular: in the
latter case $\Sp(E,D)$ resolves the double point into a handle. Now consider 
the curve $C$ along $|\zeta|\to\infty$. Either for every $R>0$
the domain
$|\zeta|>R$ contains at least one handle, or there is some $R$ for which $\Sp(E,D)$ only contains
double points. In the latter case $\Sp(E,D)$ must have two intersecting branches, one
of which can be parameterised by $\Delta = \{\zeta;|\zeta|> R\}$, thought of as a punctured
parameter disc about the point at infinity of $C$. Thus we have a map
\[
\chi:\Lambda\times\Delta\to\C^\times,\quad \chi(\cdot,\zeta)\in\Hom(\Lambda,\C^\times),
\]
which is holomorphic in $\Delta$ for each $\lambda$. The following result is a direct 
consequence of \cite[Theorem 4.1 and Lemma 5.1]{BohPP} (cf. \cite[\S 4]{Taisurv}).
\begin{thm}\label{thm:spectrum}
When $\Sp(E,D)$ has only double points along $\Delta$ there exists a function
\[
\psi:\C\times\Delta\to\H
\]
satisfying$:$
\begin{enumerate}
\item $\caD\psi = 0$ for all $\zeta\in\Delta$,
\item $\psi$ is holomorphic in $\zeta$ and $\lim_{\zeta\to\infty}\psi(z,\zeta)e^{-z\zeta} = 1$,
\item $\psi(z+\lambda,\zeta) = \mu(\lambda)^{-1}\psi(z,\zeta)\chi(\lambda,\zeta)$ for all $z\in\C$,
$\lambda\in\Lambda$.
\end{enumerate}
Further, $\psi$ is uniquely determined by $\psi(0,\zeta)$.
\end{thm}
We will call this function $\psi(z,\zeta)$ the \textit{multiplier Baker function} for $\Sp(E,D)$. 
By the properties above it has the expansion (in left Fourier series)
\[
\psi(z,\zeta) = (1 +\sum_{j=1}^\infty a_j(z)\zeta^{-j})\exp(z\zeta),\quad a_j:\C\to\H,\ |\zeta|>R.
\]
By rescaling $\zeta$ we may as well assume $\Delta$ is the punctured
disc $|\zeta|>1$. In the next section we will introduce the QKP Baker function, and 
later on we will consider under what conditions we can show that the two agree.

The following example will help us understand later (section \ref{sec:spectral}) 
the difference between the multiplier spectrum and the
QKP spectral curve. 
It comes from the study of Hamiltonian stationary Lagrangian (HSL) tori in $\R^4$. 
\begin{exam}\label{ex:HSL}
Fix a torus $\C/\Lambda$ and equip $\C$ with its standard metric $\langle z,w\rangle =
\frac{1}{2}(z\bar w + \bar z w)$. We use this to embed the dual lattice $\Lambda^*$ in $\C$.
Fix some non-zero $\beta_0\in\Lambda^*$ and define $\beta(z) = 2\pi\langle\beta_0,z\rangle$. 
Now consider the complex quaternionic line bundle $(E,J)$ where $\pi:E\to\C/\Lambda$ is the trivial right 
$\H$-bundle and $J\sigma = N\sigma$ for $N=e^{j\beta}i$. This has \qh structure
\[
D\sigma = \frac{1}{2}(d\sigma + N*d\sigma).
\]
From \cite{LesR} we know that $f:\C/\Lambda\to\H$ is HSL with Maslov
class $\beta_0$ whenever $Df=0$. It is easy to check that $\Phi=e^{j\beta/2}i$ is a parallel
section of $\pi^*\hat E$ for the unique flat Hermitian connexion $\hat\nabla$ on $\hat E$ for which
$\hat\nabla'' = \bar\partial_J$, and the Dirac operator given by \eqref{eq:Dirac} is
\begin{equation}\label{eq:HSLDirac}
\caD = \frac{\partial}{\partial\bar z} - \frac{\pi}{2}\beta_0 j.
\end{equation}
Notice that $(\hat E,J)$ is trivial if $\beta_0/2\in\Lambda^*$ but otherwise a spin bundle, 
since the monodromy of $\Phi$ is given by $\mu(\lambda) = e^{i\beta(\lambda)/2}=\pm 1$ for
$\lambda\in\Lambda$. 

We can explicitly calculate $\Sp(E,D)$ for this example 
by writing any solution of $\caD\psi=0$ in the form
\[
\psi(z) = [\mu^{-1}(z)\varphi_1(z) + j\mu^{-1}(z)\varphi_2(z)]\exp[\pi i (\xi z + \eta\bar z)],
\]
where $\mu(z) = e^{i\beta(z)/2}$ and $\varphi_m$ are $\Lambda$-periodic complex valued functions,
and $\xi,\eta$ are complex parameters which parameterise the logarithmic spectrum. 
Since the complex valued functions $\varphi_m$ are both $\Lambda$-periodic they can be 
represented by Fourier series
\[
\varphi_m = \sum_{\alpha\in\Lambda^*}\varphi_{m\alpha}e^{2\pi i\langle\alpha,z\rangle}.
\]
There is a non-trivial solution to $\caD\psi=0$ if and only if there exists
$\alpha\in\Lambda^*$ for which the linear system
\begin{equation}\label{eq:HSLalpha}
\begin{pmatrix}
\alpha + \eta +\beta_0/2 & -i\beta_0/2 \\ 
 i\bar\beta_0/2 & \bar\alpha + \xi +\bar\beta_0/2
\end{pmatrix}
\begin{pmatrix}
\varphi_{1\alpha} \\ \varphi_{2\alpha}
\end{pmatrix}
=
\begin{pmatrix} 0 \\ 0 \end{pmatrix},
\end{equation}
has a non-trivial solution. Thus if we set 
\[
F_\alpha(\eta,\xi) = (\alpha + \eta +\beta_0/2)(\bar\alpha + \xi +\bar\beta_0/2) - |\beta_0|^2/4
\]
we can describe the logarithmic spectrum $\tilde\Sigma$ of $(E,D)$ as the union of infinitely 
many irreducible rational curves:
\[
\tilde\Sigma = \cup_{\alpha\in\Lambda^*} C_\alpha,\quad C_\alpha =
\{(\eta,\xi)\in\C^2;F_\alpha(\eta,\xi)=0\}.
\]
Even though each component $C_\alpha$ is smooth the curve $\tilde\Sigma$ possesses infinitely many
singularities caused by the intersections of components.
The dual lattice $\Lambda^*$ acts on $\tilde\Sigma$ by
$(\eta,\xi)\mapsto(\eta+\alpha,\xi+\bar\alpha)$ and it is easy to see that this identifies all
components $C_\alpha$ with one, say $C_0$. Thus 
\[
\Sp(E,D)\simeq \tilde\Sigma/\Lambda^*\simeq C_0/\sim
\]
where $\sim$ is the equivalence relation on $C_0$ defined by
\[
(\eta,\xi)\sim (\eta',\xi')\ \Leftrightarrow (\eta',\xi') = (\eta+\alpha,\xi+\bar\alpha)\
\text{for some}\ \alpha\in\Lambda^*.
\]
This identification creates the singularities of $\Sp(E,D)$, which are of two types.
\begin{enumerate}
\item For each non-zero $\alpha\in\Lambda^*$ the points $(\eta,\xi)$ and 
$(\eta+\alpha,\xi+\bar\alpha)$ are identified whenever
\[
\frac{\alpha}{\bar\alpha}Z^2+\alpha Z + \frac{|\beta_0|^2}{4}=0,\quad Z=\xi+\frac{\bar\beta_0}{2}.
\]
When the discriminant (which is proportional to $|\alpha|^2-|\beta_0|^2$) is
non-zero this gives two distinct intersections between $C_0$ and $C_\alpha$,
which will both be double points provided no other component of $\tilde\Sigma$ intersects here. 
The possibility of more components meeting in
one point, and hence higher order singularities, cannot be ruled out: the conditions are rather
subtle and depend upon both $\beta_0$ and $\Lambda$. When the discriminant vanishes $C_0$ and 
$C_\alpha$ meet tangentially at one point. This introduces a cuspidal singularity. 
\item The point $(0,0)$ is identified with every point in the set
\[
S=\{(\alpha,\bar\alpha);|\alpha+\frac{\beta_0}{2}|=|\frac{\beta_0}{2}|,\ \alpha\in\Lambda^*\}.
\]
This includes $(\beta_0,\bar\beta_0)$ and therefore one of the above cusps is folded into this
singularity.
\end{enumerate}
This leads to the multiplier Baker function
\begin{equation}\label{eq:HSLmult}
\psi(z,\zeta)= (1+j\frac{\pi\bar\beta_0}{2}\zeta^{-1})\exp(-\frac{\pi^2|\beta_0|^2\bar
z}{4}\zeta^{-1})e^{z\zeta},
\end{equation}
which is uniquely determined by its initial value
$\psi(0,\zeta) = 1+j(\pi\bar\beta_0/2)\zeta^{-1}$. It follows that the space
$H^0_D(E)$ of global \qh sections is spanned by the sections $\Phi\psi$ obtained by evaluating 
$\psi$ at each $\zeta = \pi i(\bar\alpha + \bar\beta_0/2)$ for $(\alpha,\bar\alpha)\in S$ 
(cf.\ \cite[Thm 2.9]{LesR}).
\end{exam}

\section{The QKP hierarchy.}\label{sec:qkp}

\subsection{Formal construction of the QKP hierarchy.}
The QKP hierarchy is a real form of the two component KP 
hierarchy (and is referred to elsewhere, for example \cite{Tai98,Tai04}, as
the Davey-Stewartson II hierarchy).  It is constructed by the formal
dressing method, working entirely within a quaternionic framework, so that the comparison between
KP and QKP is quite literally the replacement of $\C$ by $\H$.  I will follow the purely algebraic
approach given in two papers by George Wilson \cite{Wil79,Wil81}. 

We begin by fixing a real differential algebra $\caB$, with derivation $\dy$,
of the form
\[
\caB=\R[u^{(k)}_{\alpha\beta}],\ \text{for}\ \alpha,k=0,1,2,\ldots,\ 
\beta=1,2,3,4.
\]
These generators are algebraically free but related under the derivation by
\[
\dy u^{(k)}_{\alpha\beta} = u^{(k+1)}_{\alpha\beta}.
\]
With this we construct a formal pseudo-differential operator with coefficients
in $\caB\otimes\H$ 
\begin{equation}\label{eq:L_operator}
L = i\dy + U_0 +U_1\dy^{-1} +\cdots
\end{equation}
where
\[
U_0 = j(u_{03}+iu_{04}),\quad 
U_{\alpha} = u_{\alpha 1} + iu_{\alpha 2} + j(u_{\alpha 3} + iu_{\alpha 4}).
\]
It is important for this construction that the leading coefficient $i$ of $L$
is regular for the Lie algebra structure on $\H$ (i.e.,\ the commutator of $i$
has minimal dimension 2). This means we may ignore the component of $U_0$
which commutes with $i$. 

The construction provides an infinite family of
independent derivations on $\caB$, each of which commutes with $\dy$, via Lax equations.
This is done via the formal dressing construction to produce
a subalgebra of the commutative algebra $Z(L)$ of all
pseudo-differential operators (over $\caB\otimes\H$) 
which commute with $L$. The method is summarised in the following two theorems.
\begin{thm}[\cite{Wil79}]\label{thm:dressing}
There exists a formal operator of the form
\[
K=1+\sum_{k\geq 1} a_k(i\dy^{-1})^k
\]
such that $K^{-1}LK=i\dy$. Moreover, $K$ is unique up to right multiplication
by operators of the form $1+\sum_{k\geq 1} c_k\dy^{-k}$ where each $c_k$ is
a complex constant.
\end{thm}
Note that the components of the coefficients $a_k$ do not
belong to $\caB$ but generate an extension algebra $\hat\caB$.
\begin{thm}[\cite{Wil79}]
Let $Z_0(L)$ denote the image of the $\R$-algebra homomorphism
\[
\C[\dy]\to Z(L);\ P_0\mapsto P=KP_0K^{-1}.
\]
Then the coefficients of $P$ all lie in $\caB\otimes\H$ and there is  
a derivation $\partial_P$ on $\caB$, characterised by 
\begin{equation}\label{eq:QKP}
\partial_PL = [L,P_+],\quad [\partial_P,\partial_y]=0,
\end{equation}
where $P_+$ denotes the differential operator part of $P$. For any two $P,Q\in Z_0(L)$ the
derivations $\partial_P$ and $\partial_Q$ commute and satisfy
\[
\partial_PQ_+ = \partial_Q P_+ + [Q_+,P_+].
\]
\end{thm}
We single out one of these derivations for special attention: the derivation
$\partial_L$ will be renamed $\partial_x$. 
\begin{rem}
In particular, this includes a family of p.d.e of the form
\begin{equation}\label{eq:DQKP}
\partial_PL_+ = \dx P_+ + [L_+,P_+].
\end{equation}
Historically these are the equations referred to as the hierarchy, 
since these are the equations which occur in practical applications to physics
and fluid dynamics (for example, the $t_2$ flow yields the Davey-Stewartson II equations). 
But this is misleading:
the coefficients of $P_+$ cannot in general be expressed as differential polynomials in the
coefficients of $L_+$ and therefore the equations \eqref{eq:DQKP} alone do
not carry the information contained in the definition \eqref{eq:QKP}.  
Therefore we follow Sato's nomenclature, also used in \cite{SegW}, and call the system
of equations \eqref{eq:QKP} the QKP hierarchy. 
\end{rem}
These derivations also extend to $\hat\caB$ via $\partial_PK =P_-K$ and we can
introduce the formal Baker function. In the purely algebraic setting it is a formal power series
\[
\psi =  (1+a_1\zeta^{-1}+\cdots)\exp(x\zeta+iy\zeta)   = K\exp(x\zeta+iy\zeta)
\]
where $\zeta$ is a formal parameter (to make sense of this we can take
an appropriate extension of the differential algebra $\hat\caB$).
From the definition of $L$ it follows that $L\psi =  -\psi \zeta$.
Therefore 
\[
\dx\psi + L_+\psi = L_-\psi +\psi\zeta -\psi\zeta -L_-\psi =0
\]
Now we can extend the definition of $\psi$ to 
\[
\psi = K\exp[\sum_{k\in\N}(s_k+it_k)\zeta^k]
\]
for real variables $s_k,t_k$. It is not hard to see that for every 
$P_0\in\C[\dy]$ we can find a sequence
$t=(x,y,s_2,t_2,\ldots)$ for which, with the relabelling of $\partial_P$ as
$\partial_t$, we have $\partial_t\psi + P_+\psi = 0$.

\subsection{The Grassmannian class of solutions.}

Following Segal and Wilson \cite{SegW} we can construct an infinite dimensional
Grassmannian $Gr(\H)$ whose points essentially parameterise the set of all
formal Baker functions which actually converge for $|\zeta|=1$. 
It is well-known that for the two component KP hierarchy
this class of solutions ``linearise'' on a
Grassmannian $Gr(\C^2)$ of subspaces of the Hilbert
space $H=L^2(S^1,\C^2)$. 
The QKP hierarchy is a real form of two component
KP obtained by imposing a reality condition: this is
achieved by fixing a left $\H$ action on $H$ for which $j$ acts conjugate 
linearly.

The quickest way to do this is to realise $H$ as $L^2(S^1,\H)$.
We view $\H$ as having two complex structures: the first
arises from left multiplication by $i\in\H$, the second comes from
right multiplication by $i$. These structures are inherited by
any space of $\H$-valued functions. We view $H$ as a complex vector space with
respect to the first complex structure. Then it has
the Hermitian inner product 
\[
<f,g> = \int_{S^1} (f\bar g)_\C
\]
where $\bar g$ is the quaternionic conjugate, and if $q=a+bj\in\H$
for $a,b\in\C$ then $q_\C=a$. The integral is normalised in the usual way so
that $<1,1>=1$. 
To make calculations for differential operators acting on the left, we
represent each element of $H$ in its left Fourier series
\[
f(\zeta) = \sum_{m\in\Z} (u_m+v_mj)\zeta^m,\ |\zeta|=1,\ u_m,v_m\in\C.
\]
Thus we identify $L^2(S^1,\H)$ with $L^2(S^1,\C^2)$, so that
\begin{equation}\label{eq:rep}
L^2(S^1,\C^2)\to L^2(S^1,\H);\ (u,v)\mapsto u+j\bar v,
\end{equation}
where $\bar v(\zeta) = \overline{v(\bar\zeta)}$.
Both left $\H$ multiplication and right complex multiplication
preserve the polarization of $H$ into
orthogonal subspaces $H_+$ and $H_-$ which consist
of functions whose left Fourier series have,
respectively, only non-negative and only negative powers of $\zeta$.

Now define $Gr(\C^2)$ to be the space of all complex 
subspaces $W\subset H$ (i.e., $iW = W$) for which the
projections $\mathrm{pr}_+:W\to H_+$ and $\mathrm{pr}_-:W\to H_-$ are respectively Fredholm (of
index zero) and Hilbert-Schmidt. Then $Gr(\C^2)$ is a complex Hilbert manifold \cite{PreS}.
Inside this we have the real submanifold
\[
Gr(\H) = \{W\in Gr(\C^2); jW=W\},
\]
the fixed point subspace of the real involution $W\mapsto jW$.
As with the KP equations, we
can produce an abelian subgroup whose action on $Gr(\H)$ corresponds to the QKP equations.

Define $\Gamma\subset C^\omega(S^1,\C^*)$ to be the abelian subgroup
of all non-vanishing analytic functions with winding number zero, and let $\caG$ denote its Lie
algebra $C^\omega(S^1,\C)$. $\Gamma$ acts $\H$-linearly on $H$ by
\[
\gamma:H \to H;\ f \mapsto f\gamma,
\]
Thus we have a representation $\Gamma\subset GL_{\res}(\H)$,
hence $\Gamma$ acts on $Gr(\H)$ as a real group: to emphasize the definition
we will write this action as $\gamma\circ W= W\gamma$. The group
$\Gamma$ factorises into the product $\Gamma_-.\Gamma_+$ of two subgroups:
$\Gamma_+$, whose elements extend
holomorphically into the disc $|\zeta|<1$ and are unimodular at $\zeta=0$, and
$\Gamma_-$, whose elements extend holomorphically into the disc $|\zeta|>1$
and take a real positive value at $\zeta=\infty$. (This slightly unusual normalisation
for $\Gamma_-$ and $\Gamma_+$ makes our discussion of the $\Gamma$-orbits easier later on, since
the real scaling action is trivial but the action of the unimodular scaling plays an important role.)
We write $\caG = \caG_++\caG_-$ for the 
corresponding Lie algebra splitting.  We will parameterise elements of $\Gamma_+$ by
writing each in the form
\[
\gamma(\bt) = \exp[it_0+\sum_{k\geq 1} (s_k+it_k)\zeta^k],\ \bt=(s_1,t_1,\ldots)
\in \caG_+.
\]
We will tend to use $x,y$ instead of $s_1,t_1$ below, and write $z=x+iy$.

In the same manner as \cite{SegW}, we can assign to $W\in Gr(\H)$
a convergent Baker function and thereby obtain solutions to the QKP equations.
Such an assignment only works when the $\Gamma_+$-orbit of $W$ meets 
the big cell, i.e., the open dense subset of $Gr(\C^2)$ 
consisting of all $W\in Gr(\C^2)$ for which $\mathrm{pr}_+:W\to H_+$ is invertible. 
The following result shows that, like the KP case, this condition is always
satisfied. 
\begin{thm}\label{thm:bigcell}
Let $W\in Gr(\H)$ and define
\[
\Gamma_1 = \{\exp((x+iy)\zeta);x,y\in\R\}\simeq\R^2.
\]
Then the $\Gamma_1$-orbit of $W$ meets the big cell off a real analytic
(proper) subvariety of $\R^2$. Consequently the $\Gamma_+$-orbit of $W$ meets the big cell on the
complement of a real analytic $($proper$)$ subvariety of $\Gamma_+$.
\end{thm}
The proof is a variation on the proof of the corresponding result in \cite{SegW} and
will be omitted here since this theorem does not play a central part in the subsequent discussion.
As a consequence, to every $W\in Gr(\H)$ we can assign a convergent Baker 
function $\psi_W$ as
follows. Since $\Gamma_+$-orbits meet the big cell on open sets
there is a function, defined for almost all $\bt\in\caG_+$, 
\begin{equation}\label{eq:Bakerseries}
\psi_W(\bt) = (1 + a_1(\bt)\zeta^{-1} +\cdots)\gamma(\bt),
\end{equation}
taking values in $W$: it is characterised by the property that 
\[
\mathrm{pr}_+(\psi_W(\bt)\gamma(\bt)^{-1}) = 1.
\]
Moreover, $\psi_W$ uniquely determines $W$, since the set $\psi_W(\mathbf{0}),
\psi_W'(\mathbf{0}),\psi_W''(\mathbf{0}),\ldots$ of all $y$ derivatives spans 
an $\H$-subspace of $H$ whose closure is $W$.

From now on we will set $z=x+iy$ and, for notational convenience,
we will use, for example, $\gamma(t_0)$ to denote setting every parameter 
except $t_0$ equal to zero.
\begin{thm}
To every $W\in Gr(\H)$ we can assign a formal pseudo-differential operator
$L_W$ of the form \eqref{eq:L_operator} satisfying $L_W\psi_W=-\psi_W\zeta$.
Consequently, for $t=s_k$ or $t=t_k$, $k\in\N$ there exists a differential operator 
$P_+$ for which
\[
\partial\psi_W/\partial t + P_+\psi_W=0,\quad\text{hence}\ L_t = [P_+,L].
\]
In particular, for each $W$ we obtain a Dirac operator 
$\caD = \partial/\partial\bar z +U_W$
for which $U_W=-(a_1+ia_1i)/2$, $\caD\psi_W=0$ and the equations for
$\partial U_W/\partial t$ are \eqref{eq:DQKP}.
\end{thm}
The proof is identical to that for the KP hierarchy given in \cite{SegW}. 
In particular, from $\psi_W$ we extract
\[
\tilde\psi_W(\bt) = 1 + \sum_{k>0} a_k(\bt)\zeta^{-k},
\]
from which we obtain $L_W = K_Wi\dy K_W^{-1}$ using
\[
K_W = 1 + \sum_{k>0} a_k(\bt)(i\dy^{-1})^k.
\]
\begin{rem}\label{rem:KP}
Let $Gr(\C)$ to denote the Segal-Wilson Grassmannian for 
$L^2(S^1,\C)$. We can embed this in $Gr(\H)$ using 
$V\mapsto V\oplus \bar V$, where the latter is the
space $\{ (u,\bar v);u,v\in V\}\subset Gr(\C^2)$. Points of $Gr(\H)$ of
this type yield solutions to the (complex scalar) KP hierarchy of equations, since it
is clear that the complex Baker function $\psi_V$ which Segal and
Wilson assign to
$V\in Gr(\C)$ is also our quaternionic Baker function $\psi_W$, for
$W=V\oplus \bar V$.  Therefore all calculations reduce to those of
\cite{SegW}. We will use $Gr_\KP$ to denote the image of $Gr(\C)$ in
$Gr(\H)$ and denote its complement by $Gr_\QKP$. We can characterise the points of $Gr_\KP$ as
follows.
\begin{lem}\label{lem:KP} 
$W\in Gr_\KP$ if and only if $Wi=W$, equally, if and only if the Dirac potential $U_W$ is
trivial. 
\end{lem}
\begin{proof}
If $W\in Gr_\KP$ then clearly we have both $Wi=W$ and $U_W=0$. Now if $Wi=W$ then $W=V+jV$
where 
\[
V = \{\frac{1}{2}(f- ifi);f\in W\}.
\]
In particular, $\frac{1}{2}(\psi_W-i\psi_Wi)$ belongs to $V$ and by uniqueness of the 
Fourier expansion for
Baker functions must equal $\psi_W$, hence $\psi_W$ takes values in $V$. This means 
$V\in Gr(\C)$, since derivatives of $\psi_W$ generate $V$ over $\C$. Thus
$W\in Gr_\KP$. 

Now suppose $U_W=0$ and set $\partial = \partial/\partial z$. Then $\bar\partial\psi_W = 0$ and 
so if we write $\psi_W=\psi_1+j\psi_2$, where both $\psi_1,\psi_2$ commute with $i$,
we have $\bar\partial\psi_1=0$ and $\partial\psi_2=0$. We notice that, restricting $\psi_W$ to $x,y$,
\[
\psi_1 = (1+p_1\zeta^{-1}+\cdots)\exp(z\zeta),\quad \psi_2 = (q_1\zeta^{-1}+\cdots)\exp(z\zeta),
\]
for some complex valued functions $p_1,q_1,\ldots$. Now $\partial\psi_2=0$ means
\[
(q_1+(\partial q_1+q_2)\zeta^{-1}+\cdots+(\partial q_k+q_{k+1})\zeta^{-k}+\cdots)e^{z\zeta}=0.
\]
So $\partial\psi_2=0$ if and only if $q_k=0$ for all $k$, i.e., $\psi_2=0$. Therefore
$\psi_W=\psi_1$ and since $\bar\partial\psi_1=0$ the closure $V$ of the complex subspace of 
$W$ generated
by $\psi_W(0),\partial\psi_W(0),\ldots$ belongs to $Gr(\C)$. But $V$ generates $W$ over $\H$ and
therefore $W=V+jV$, which is $V\oplus\bar V$ in our notation. 
\end{proof}
\end{rem}

\subsection{QKP flows.}
An important consequence of the construction is that the action of $\Gamma_+$
generates the QKP flows, in the following very precise sense. Let $\bt,\bt'\in\caG_+$.
Then, by comparing Fourier series,
we deduce that 
\begin{equation}\label{eq:flow1}
\psi_{W\gamma(\bt')}(\bt) = \psi_W(\bt-\bt')\gamma(\bt'),
\end{equation}
and therefore, treating $\bt$ as variable and $\bt'$ as a constant,
we deduce
\begin{equation}\label{eq:flow2}
L_{W\gamma(\bt')}(\bt) = L_W(\bt-\bt').
\end{equation}
Now, if our interest is purely in the QKP operator $L_W$
then, as with the KP hierarchy \cite{SegW}, we observe that for 
any $\gamma\in\Gamma_-$ we have
$\psi_{W\gamma} = \psi_W\gamma$, hence $L_{W\gamma}=L_W$.
In fact we have the following result.
\begin{thm}\label{thm:phasespace}
The quotient  space $\caM=Gr(\H)/\Gamma_-$ is a manifold, and the
map
\[
\caM \to \{L_W;W\in Gr(\H)\};\quad \Gamma_-\circ W\mapsto L_W
\]
is bijective.
\end{thm}
\begin{proof}
It is easy to extrapolate from the proof of
the analogous result in Segal and Wilson \cite[2.4]{SegW} that the subgroup of elements
of $\Gamma_-$ of the form $1+O(\zeta^{-1})$ acts freely on $Gr(\H)$, while the constant scalars act
trivially, hence $\caM$ is a manifold.
By Theorem \ref{thm:dressing} $L_W$ uniquely determines $\psi_W$ up to
right multiplication by an element of $\Gamma_-$, and $\psi_W$ 
determines $W$, so the map is bijective.
\end{proof}
Since $\Gamma$ is abelian the group $\Gamma_+$ acts on $\caM$ and it 
is clear from \eqref{eq:flow2} that
the orbits correspond to the flows of the QKP hierarchy. 
The action of the circle subgroup $\Gamma_0=\{\exp(it_0);t_0\in\R\}\subset\Gamma_+$ 
turns out to be important. In the first place, the next theorem shows that its orbits 
are either points or circles
depending upon whether or not $W\in Gr_\KP$, and this distinction descends to the
disjoint union $\caM=\caM_\KP\cup\caM_\QKP$ obtained by taking the quotient of $Gr_\KP\cup
Gr_\QKP$ by the action of $\Gamma_-$.

For ease of notation, for any $W\in Gr(\H)$ set $W(t_0,z) = We^{it_0+z\zeta}$,
with the usual convention that the absence of either variable denotes that it
is zero.
\begin{thm}\label{thm:periodic}
\begin{enumerate}
\item For any $W\in Gr(\H)$
\begin{equation}\label{eq:L_W}
L_{W(t_0)} = e^{-it_0}L_We^{it_0}.
\end{equation}
\item The group $\Gamma_0$ acts trivially on $\caM_\KP$, while
the space $\caM_\QKP/\Gamma_0$
of $\Gamma_0$-orbits in $\caM_\QKP$ is the quotient of $\caM_\QKP$ by a free
action of $S^1$.\\
\item Let $\Lambda\subset\C$ be a lattice, then
$L_W(z+\lambda) = \mu(\lambda)^{-1}L_W(z)\mu(\lambda)$ for some
$\mu\in\Hom(\Lambda,S^1)$ if and only if the map
\begin{equation}\label{eq:periodic}
\C\to\caM/\Gamma_0 ;\quad z\mapsto [L_{W(z)}]
\end{equation}
is $\Lambda$-periodic, where $[L_W]$ denotes the $\Gamma_0$-orbit of
$L_W\in\caM$.
\end{enumerate}
\end{thm}
\begin{proof} (a) First we observe that
\[
e^{-it_0}\psi_W(z)e^{it_0} = 
(1 + e^{-it_0} a_1(z)e^{it_0} \zeta^{-1} + \cdots)\exp(z\zeta),
\]
and that this takes values in $W(t_0)$. 
By the uniqueness of the Baker function, we deduce that
$\psi_{W(t_0)}(z) = e^{-it_0}\psi_W(z)e^{it_0}$. The formula
\eqref{eq:L_W} follows at once.

\noindent
(b) 
The definition of the $\Gamma_+$ action, combined with \eqref{eq:L_W}, yields
\[
e^{it_0}\circ L_W = L_{W(t_0)}= e^{-it_0}L_We^{it_0}.
\]
This is trivial if $L_W$ is purely complex, which by Lemma \ref{lem:KP} is the case precisely when
$W\in Gr_\KP$, so $\Gamma_0$ acts trivially on $\caM_\KP$.
Now observe that for any $q\in\H$, $e^{-i\pi}qe^{i\pi} = q$.
Therefore if we define an $S^1$ action on $\caM_\QKP$ by 
$e^{it_0}\cdot L_W = L_{W(t_0/2)}$ this action is free, since
the points in the subspace $\caM_\QKP$ are those
for which $\psi_W$ is not purely complex. Clearly the
quotient of $\caM_\QKP$ by this action equals $\caM_\QKP/\Gamma_0$.

\noindent
(c) Suppose $L_W(z+\lambda) = \mu(\lambda)^{-1}L_W(z)\mu(\lambda)$ for some
$\mu\in\Hom(\Lambda,S^1)$ and all $z\in\C$. By \eqref{eq:flow2} and
\eqref{eq:L_W} this means
\[
L_{W(z+\lambda)}(z') = L_{W(z)}(z'-\lambda) = L_{W(z)\mu(-\lambda)}(z')
\]
thinking of $\mu(\lambda)\in\Gamma_0$. But the last has the same
$\Gamma_0$-orbit as $L_{W(z)}$, so $[L_{W(z+\lambda)}]=[L_{W(z)}]$.

Conversely, suppose that \eqref{eq:periodic} is $\Lambda$-periodic. Let 
\[
G = \{\exp(it_0+z\zeta)\in\Gamma\}=\Gamma_0\Gamma_1,
\]
and let $\caS\subset\caM$ be the $G$-orbit of $L_W$. The projection $\caM\to\caM/\Gamma_0$ 
makes $\caS$ a bundle over the $\Gamma_1$-orbit $M\simeq\C/\Lambda$ of $[L_W]$. When 
$L_W\in\caM_\KP$ this
projection is a bijection, so that $L_W(z)$ is $\Lambda$-periodic and $\mu$ is trivial. Otherwise
$\caS\to M$ is an $S^1$-bundle, by (b), with a natural flat connexion for which the action of
$\Gamma_1\subset G$ is horizontal. This makes $z\mapsto L_{W(z)}$ a flat section over the universal
cover $\C$ of $M$ and hence it has a monodromy $\mu'\in\Hom(\Lambda,S^1)$. The action of $S^1$ is via
$\Gamma_0$ and therefore
\[
L_{W(z+\lambda)} = \mu'(\lambda)^{-1}L_{W(z)}\mu'(\lambda),
\]
which implies
\[
L_W(z'-\lambda)  = \mu'(\lambda)^{-1}L_W(z')\mu'(\lambda).
\]
Taking $\mu = (\mu')^{-1}$ gives the required result.
\end{proof}

\subsection{Solutions of finite type.}

By adapting the construction in \cite{SegW} we can construct many points $W\in
Gr(\H)$ corresponding to spectral data and thereby construct Baker functions
(using, for example, Riemann $\theta$-functions). Our spectral data will be a collection
$(X,\rho,P,\zeta,\caL,\varphi)$ of the following.
\begin{enumerate}
\item $X$ is a complete, reduced, algebraic curve of arithmetic genus $g$,
with fixed-point free anti-holomorphic involution $\rho$. $X$ need not be
irreducible, but if it is reducible it must have no more than two irreducible
components and these must be swapped by $\rho$.
\item $P\in X$ is a smooth point 
and $\zeta^{-1}$ is a local parameter about $P$,
\item $\caL$ is a holomorphic line bundle of degree $g+1$ for which
$\overline{\rho^*\caL}\simeq\caL$  and $\caL(-P-\rho P)$ is non-special.
This induces a unique, up to sign, conjugate linear isomorphism 
$\bar\rho^*:\caL\to\caL$ covering $\rho$
and satisfying $(\bar\rho^*)^2=-1$ (so $\bar\rho^*$ is a quaternionic involution).
\item $\varphi$ is a holomorphic
trivialising section of $\caL$ over the disc
$\Delta_P=\{Q;|\zeta(Q)^{-1}|< 1\}$ and its boundary circle $C_P$, both of 
which we assume contain no singular points of $X$.
\end{enumerate}
\begin{rem}\label{rem:torsionfree}
(i) There is no requirement that $X$ be smooth. In general we let $X^\sm$ denote the 
open variety of smooth points on $X$.
It is possible for $X$ to be disconnected but Example
\ref{ex:KPcurves}  shows that this only leads to points in $Gr_\KP$. 
When $X$ is singular the condition (c) is more strict than necessary. 
As explained in \cite[p.38]{SegW}, $\caL$ need only be a 
maximal torsion free coherent sheaf of rank $1$ with $\chi(\caL)=2$. \\
(ii) The QKP spectral curve is the natural quaternionic analogue of the pointed curve which
appears in KP theory. It seems highly likely that $Gr(\H)$ 
plays the same role for the moduli space of such curves that $Gr(\C)$
plays for the moduli space of all pointed complete irreducible algebraic curves.
\end{rem}
From this data we construct a point $W\in Gr(\H)$: the points constructed this way will be called
\emph{of finite type} (this terminology is justified by Theorem \ref{thm:finitetype} below). 
To do this, first identify $S^1$ with the circle $C_P$ (and also with $\rho C_P$ by the map
$\zeta\mapsto \overline{\rho^*\zeta}$). Let $X_0$ denote the closed
non-compact surface $X\setminus (\Delta_P\cup\rho(\Delta_P))$.
We define
\begin{equation}\label{eq:w}
w:H^0(X_0,\caL)\to H;\ \sigma\mapsto (\sigma -j\bar\rho^*\sigma)/\varphi.
\end{equation}
This is clearly $\C$-linear.  Now define $W$ to be
the closure of the image of $w$. A simple generalisation of the Mayer-Vietoris
argument in \cite[\S 6]{SegW}
shows that $W\in Gr(\C^2)$. Observe that since $\bar\rho^*$ is
quaternionic on $H^0(X_0,\caL)$ we have
\[
jw(\sigma) = w(\bar\rho^*\sigma)
\]
and therefore $W\in Gr(\H)$.  The non-speciality condition on
$\caL(-P-\rho P)$ ensures that $W$ lies in the big cell (although this is not
essential we may as well assume this since it happens in the $\Gamma_+$-orbit
by Theorem \ref{thm:bigcell}).

The action of $\Gamma$ on the spectral data is easily
calculated. Clearly
\[
w(\sigma)\gamma^{-1} = (\sigma -j\bar\rho^*\sigma)/(\varphi\gamma).
\]
This twists the trivialisation $\varphi$ into $\varphi\gamma$, which 
we interpret as trivialisation for the line bundle $\caL\otimes \ell(\gamma)$. Here 
$\ell(\gamma)$ is a degree zero line bundle obtained by glueing the 
trivial bundles over $X_0$ and $\mathrm{clos}(\Delta_P)$ together using $\gamma$ as a 
transition function, and glueing that to the trivial bundle over 
$\mathrm{clos}(\rho(\Delta_P))$ using $\bar\gamma$ (after identifying $\zeta$
with $\overline{\rho^*\zeta}$). Since we will need to work with these trivialisations
later on we will be more precise about this by introducing better notation (and then
suppress this in most of what follows until we require it).

The construction of $\ell(\gamma)$ equips it with trivialising sections $\tau_P(\gamma)$ and
$\tau_0(\gamma)$, over $\Delta_P\cup\rho\Delta_P$ and $X\setminus\{P,\rho P\}$ respectively,
satisfying the transition relation
\begin{equation}\label{eq:transition}
\tau_P(\gamma)= \gamma\tau_0(\gamma)\ \text{over}\ \Delta_P\setminus\{P\},
\end{equation}
and the $\rho^*$-conjugate of this about $\rho P$. So by ``$\varphi\gamma$'' we really mean the
trivialisation $\varphi\tau_P(\gamma)$. Using this we obtain
\[
H^0(X_0,\caL\otimes\ell(\gamma))\to C^\omega(S^1,\C);\quad \sigma\mapsto
\sigma/\varphi\tau_P(\gamma).
\]
Now we observe that
\begin{equation}\label{eq:varphigamma}
\frac{\sigma}{\varphi\tau_P(\gamma)} = \frac{\sigma\tau_0(\gamma)^{-1}}{\varphi}\gamma^{-1},
\end{equation}
and $\sigma\tau_0(\gamma)^{-1}\in H^0(X_0,\caL)$, hence $W\gamma^{-1}$ corresponds to replacing
$\caL$ by $\caL\otimes\ell(\gamma)$ and $\varphi$ by $\varphi\tau_P(\gamma)$. 
We have therefore proved the following lemma.
\begin{lem}\label{lem:twist}
If $W$ corresponds to the spectral data $(X,\rho,P,\zeta,\caL,\varphi)$ and
$\gamma\in \Gamma$ then
$W\gamma^{-1}$ corresponds to the same data with $\caL$
and $\varphi$ replaced by $\caL\otimes \ell(\gamma)$ and $\varphi\gamma$.
\end{lem}
Notice that when
$\gamma\in\Gamma_-$ the two line bundles are isomorphic and only the trivialisation changes.

A consequence of this construction is the epimorphism of real groups
$\ell:\Gamma\to J_\R(X)$, where the target here is the connected
component of the identity of the real subgroup
\[
\{L\in \Jac(X);\overline{\rho^*L}\simeq L\}
\]
of the Jacobi variety of $X$. The restriction of $\ell$ to
$\Gamma_+$ is still onto.

We can also characterise the Baker function in terms of the spectral data.
Since $\caL(-P-\rho P)\otimes\ell(\gamma)$ is almost always non-special (by Theorem 
\ref{thm:bigcell}) for almost all $\gamma\in\Gamma$ there is a unique $\sigma_\gamma\in
H^0(X,\caL(-\rho P)\otimes\ell(\gamma))$ for which $\sigma_\gamma/(\varphi\gamma)$ has
value $1$ at $P$: then
$\sigma_\gamma/(\bar\rho^*\varphi\bar\gamma)$ is locally holomorphic about $\rho P$
with a simple zero there. It follows from these properties that 
$\psi_W(\gamma)=w(\sigma_\gamma)$. 

\begin{rem}\label{rem:global}
This construction is at its most concrete in the case where the trivialising 
section $\varphi$ extends to a globally holomorphic section on the whole 
of $X$. Such a choice is always possible for our spectral data (for example, using a non-zero
section vanishing at $\rho P$) and we can choose $\Delta_P$ so that it contains no zeroes of $\varphi$.
Then $\psi_1 = \sigma_\gamma/\varphi$ and $\psi_2=\bar\rho^*\sigma_\gamma/\varphi$
are both meromorphic functions on $X\setminus\{P,\rho P\}$. 
It follows that $\psi_W$ extends to $X\setminus\{P,\rho P\}$ and we have
\begin{equation}\label{eq:global}
\psi_W:\Gamma_+\times X\setminus\{P,\rho P\}\mapsto \H;\ \psi_W = \psi_1-j\psi_2.
\end{equation}
Notice that these conditions on $\varphi$ do not uniquely specify it. For if $D$ denotes its divisor 
of zeroes we are free to multiply $\varphi$ by a rational function on $X$ which does not vanish at
$P$ and whose divisor of poles lies
in the linear system $|D|$.  Thus we cannot talk of a unique global Baker 
function, but any two differ by such a rational function. Nevertheless, in what follows
we will use the phrase ``the global Baker function'' to refer to any one of these.
\end{rem}
\begin{rem}\label{rem:change_of_triv}
Every other trivialisation $\varphi'$ of $\caL$ over $\Delta_P$ is of the form $s\varphi$,
where $s$ is a non-vanishing holomorphic function on $\Delta_P$. This 
can be factorised into $s=\gamma'e^{it_0}$ where $\gamma'\in\Gamma_-$.
By previous observations we conclude that this change of trivialisation has
the effect
\[
\psi_W\mapsto e^{-it_0}\psi_We^{it_0}\gamma'.
\]
It follows that the quotient map $Gr(\H)\to\caM$ amounts to discarding from the spectral data
almost all the information
given by the choice of trivialisation: the remainder being the identification it gives between
fibres of $\caL$ over $P$ and $\rho P$. The further quotient  $\caM\to\caM/\Gamma_0$ 
discards even this information. In particular, when dealing with QKP solutions we may as well work
with global Baker functions.
\end{rem}
\begin{exam}\label{ex:KPcurves}
Consider taking $X=Y\cup\bar Y$, a disjoint union of
two copies of the same irreducible complete algebraic curve, with $\bar Y$
equipped with the opposite complex structure to $Y$. Then
$\rho:(P_1,P_2)\to (P_2,P_1)$ is a
fixed point free anti-holomorphic involution. It is not hard to see that if
we equip $Y$ with the spectral data described in \cite{SegW}, i.e., fix
$P\in Y$ with local parameter $\zeta^{-1}$, a line bundle $\caL_Y$ over $Y$ of
degree equal to the genus of $Y$ and a local trivialisation $\varphi$ about
$P$, then
we automatically obtain our quaternionic spectral data, with $\caL$ given by
$\caL|Y = \caL_Y$ and $\caL|\bar Y = \rho^*\caL_Y$. It follows that the point
$W\in Gr(\H)$ we obtain has the form $W=V\oplus \bar V$, where $V$ is the
point in the Grassmannian $Gr(\C)$ built from the data on $Y$.
Therefore, by Remark \ref{rem:KP}, spectral curves of this type only result 
in solutions to the KP hierarchy.
\end{exam}

Now let us consider the $\Gamma_+$-orbits in $\caM_\QKP$ and
$\caM_\QKP/\Gamma_0$.
Let $X_\fp$ be the singular curve obtained by identifying
the two points $P$ and $\rho(P)$ together to form an ordinary double point
and define $J_\R(X_\fp)$ to be the connected component of the identity of the
real group
\[
\{L'\in \Jac(X_\fp);\overline{\rho^*L'}\simeq L'\}.
\]
\begin{lem}
$J_\R(X_\fp)$ is an $S^1$-bundle over $J_\R(X)$.
\end{lem}
\begin{proof}
Think of $L'$ as $L\in \Jac(X)$ equipped with an isomorphism
between the fibres $L_P$ and $L_{\rho(P)}$. An isomorphism
between $\overline{\rho^*L'}$ and $L'$ is a conjugate linear
isomorphism of $\rho^*L$ with $L$ which identifies the fibre isomorphisms.
Therefore, given a fixed isomorphism
$\overline{\rho^*L}\simeq L$ we can only vary the fibre
identification by a unimodular scaling, hence the fibres of $J_\R(X_\fp)$ are
free $S^1$-orbits.
\end{proof}
As a corollary to Remark \ref{rem:change_of_triv} and the previous lemma we see how the
real groups $J_\R(X_\fp)$ and $J_\R(X)$ sit geometrically with respect to the QKP phase space
$\caM_\QKP$.
The equivalent theorem for the KP hierarchy is well-known and can be deduced
from \cite{SegW}.
\begin{thm}\label{thm:finitetype}
Suppose $W\in Gr_\QKP$ is of finite type, then
the $\Gamma_+$-orbit of $L_W\in\caM_\QKP$ can
be identified with the real group $J_\R(X_\fp)$, while the $\Gamma_+$-orbit of
$[L_W]\in\caM_\QKP/\Gamma_0$ can be identified with $J_\R(X)$. Moreover, 
$L_W$ is of finite type if and only if it admits only 
finitely many independent non-stationary QKP flows.
\end{thm}
The last claim in this theorem is explained in \S\ref{subsec:QKPcurve} where we briefly discuss the
reconstruction of the spectral curve from the QKP operator $L_W$ via its ring of commuting
differential operators.
\begin{rem}
It is  natural to ask whether there are any $W\in Gr_\QKP$ whose
$\Gamma_+$-orbit lies entirely in the big cell. When $X$ 
is smooth and has genus $g\leq 2$ there are always QKP solutions which are
globally defined.  To see this, we need to find $\caL$ of degree $g+1$ for which 
$\caL(-P-\rho P)\otimes L$ is non-special for all $L\in J_\R(X)$. This is trivial for $g=0$. More
generally, this $J_\R(X)$-orbit lies in the real slice $\Pic_{g-1}(X)^\rho$
of $\Pic_{g-1}(X)$. For $g=1$ the orbit must avoid the unique special line
bundle, namely the trivial bundle. Since $\Pic_{g-1}(X)^\rho$ has two connected
components for $g=1$ we can take any $\caL(-P-\rho P)$ to be any point lying
on the component not containing the trivial bundle. For $g=2$ the special line
bundles of degree $g-1$ lie in the image of the map $X\to \Pic_{g-1}(X)$ which
sends $Q$ to $\caO_X(Q)$. This image cannot intersect the real slice
$\Pic_{g-1}(X)^\rho$, for if $\caO_X(Q)\simeq \caO_X(\rho Q)$ then $Q=\rho Q$,
but $\rho$ has no fixed points. Hence for $g=2$ any choice of real line bundle
$\caL$ of degree $3$ provides global solutions to the QKP hierarchy.
\end{rem}
We finish this section by discussing the $\Gamma_1$-orbits. 
Let us identify $\Gamma_1$ with $\C$ by $e^{z\zeta}\mapsto z$. We wish to characterise
the homomorphism $\ell:\C\to J_\R(X)$. This is 
straightforward: $\ell(z)$ is obtained from the trivial bundle by twisting it by
$e^{z\zeta}$ about $P$ and $\overline{\rho^*e^{z\zeta}}$ about $\rho(P)$. It
follows that $\ell$ is uniquely determined by the property that 
\begin{equation}\label{eq:ell}
\frac{\partial \ell}{\partial z}|_{z=0} = 
\frac{\partial\caA_P}{\partial\zeta^{-1}}|_{\zeta^{-1}=0}
\end{equation}
where $\caA_P:X^\sm\to \Jac(X)$ is the Abel map with base point $P$ and we interpret this equation by
identifying $T_eJ_\R(X)^\C$ with $T_e\Jac(X)\simeq T^{1,0}_e\Jac(X)$. 
As a corollary to this and Theorem \ref{thm:periodic} we deduce the following.
\begin{thm}\label{thm:ell-periodic}
Suppose $W\in Gr_\QKP$ arises from spectral data and its 
$\Gamma_1$-orbit lies entirely
in the big cell. Then there is a lattice $\Lambda\subset\C$ for
which $L_W(z+\lambda) = \mu(\lambda)^{-1}L_W(z)\mu(\lambda)$ for
some $\mu\in\Hom(\Lambda,S^1)$ if and only if $\ell(z)$ is $\Lambda$-periodic.
\end{thm}
\begin{rem}\label{rem:S}
This monodromy $\mu\in\Hom(\Lambda,S^1)$ corresponds to a flat $S^1$-bundle $\caS$ over
$\C/\Lambda$, the principal $S^1$-bundle for the dual $L^*$ to our quaternionic holomorphic curve
$L$. But we can also interpret this as follows. Using $\ell$ we can pull back the $S^1$-bundle
$J_\R(X_\fp)$ to an $S^1$-bundle over $\C/\Lambda$, which is a Lie group and isomorphic to $\caS$. 
This comes equipped 
with a flat connexion as follows. By the construction above we have a natural homomorphism of 
real groups
\[
\Gamma_+\to J_\R(X_\fp)\to J_\R(X).
\]
When this is restricted to $\Gamma_1$ it gives $\ell$ and a lift $\ell^\fp:\C\to\caS$ on its
universal cover. This determines a flat connexion on $\caS$ with monodromy which is, by Theorem
\ref{thm:periodic}, $\mu$. 
\end{rem}

\section{Construction of quaternionic holomorphic curves.}\label{sec:cons}

\subsection{Construction via the Baker function.}
In this section we will suppose that $W\in Gr(\H)$ has been
chosen so that the $\Gamma_1$-orbit of $[L_W]$ is a torus $\C/\Lambda$: let
$\mu\in\Hom(\Lambda,S^1)$ be the corresponding monodromy of $L_W$. In that case there is a \qh line
bundle $(E,J,D)$ over $\C/\Lambda$ whose Dirac operator $\caD$ has Dirac potential $U_W$.

For any left $\H$-linear form $\alpha \in W^*$ the function $\psi:\C\to\H$; $\psi=\alpha(\psi_W)$
lies in the kernel of $\caD$. It corresponds to a section in $H^0_D(E)$ whenever
$\psi(z+\lambda)=\mu(\lambda)^{-1}\psi(z)$ for all
$z\in\C,\lambda\in\Lambda$. The dimension of $H^0_D(E)$ will be bounded below by
the number of independent $\alpha \in W^*$ possessing the same monodromy. In particular, with
sufficiently many of these we obtain \qh immersions of $\C/\Lambda$ in $\HP^n$.

Equivalently, to obtain a coordinate free perspective, let $V\subset W$ be an 
$\H$-subspace of $\H$-codimension $n+1$, and suppose it has the two properties:
\begin{enumerate}
\item $\psi_W(z)$ does not belong to $V$ for any $z$ (we will say $V$ is \emph{well-positioned}
in $W$),
\item for every $\lambda\in\Lambda$ we have
\[
\psi_W(z+\lambda)-\mu(\lambda)^{-1}\psi_W(z)\in V, \text{ for all } z\in\C.
\]
\end{enumerate}
Then we may define a map
\begin{equation}\label{eq:f}
f:\C/\Lambda\to \P(W/V)\simeq\HP^n;\quad z+\Lambda\mapsto [\psi_W(z)+V],
\end{equation}
where the square brackets denote the corresponding left $\H$-line in $W/V$.
\begin{thm}\label{thm:f}
The map $f:\C/\Lambda\to \HP^n$ given by \eqref{eq:f} 
is a quaternionic holomorphic curve.
\end{thm}
The proof is just an instance of the general principle linking \qh curves to \qh sections of the
dual bundle, but we will give it here since it makes explicit the complex structure $J$ on $L$.
\begin{proof}  
Over $\C$ the map $f$ has a global lift $\tilde f =
\psi_W+V$: we can think of this as a section of the pullback to $\C$ of
the line bundle $L$ corresponding to $f$. This has complex structure
$J$ for which $J\tilde f = i\tilde f$. Now
\[
*\delta f - \delta f\circ J:\tilde f \mapsto 
-2id\bar z\frac{\partial \tilde f}{\partial\bar z}\mod L.
\]
But since 
$\partial\psi_W/\partial\bar z = -U_W\psi_W$ this is identically zero, whence
$f$ is a  quaternionic holomorphic curve.
\end{proof}
\begin{rem}\label{rem:gamma}
For every $\gamma\in\Gamma_-$ we know that $\psi_{W\gamma} =
\psi_W\gamma$. Therefore the \qh curve determined by $(W,V)$ is congruent
(i.e., equivalent up to the right action of $\P GL(n+1,\H)$ on $\HP^n$, created by the
choice of isomorphism $W/V\simeq\H^{n+1}$) to
the one determined by $(W\gamma,V\gamma)$ when $\gamma\in\Gamma_-$. Similarly,
since $\psi_{We^{it_0}} = e^{-it_0}\psi_We^{it_0}$ 
the curve \eqref{eq:f} is unchanged, up to congruence of $\HP^n$, 
by the action of $\Gamma_0$.
\end{rem}
To study the periodicity conditions let us first consider the more general map
\begin{equation}\label{eq:fGamma}
f:\Gamma\to \P(W/V);\quad \gamma\mapsto[\psi_W(\gamma) + V].
\end{equation}
Here we have extended $\psi_W$ to a function on $\Gamma$ by identifying
$\Gamma_+$ with $\Gamma/\Gamma_-$, i.e., $\psi_W$ is constant on $\Gamma_-$ cosets. Now let
\[
\Gamma_V = \{\gamma\in\Gamma;W\gamma = W,V\gamma = V\}.
\]
Then $\Gamma_V$ acts on $W/V$ and we define $\Gamma^\P_V\subset\Gamma_V$ to be the subgroup of
those elements which fix $\P(W/V)$ pointwise. For example, when $W=H_+$ and $V=H_+\zeta^2$ we have
\[
\Gamma_V = \{a_0+a_1\zeta+\cdots\in\Gamma;a_0\in\C^\times\},\quad \Gamma^\P_V
=\{a_0+a_2\zeta^2+\cdots\in\Gamma;a_0\in\R^\times\}.
\]
\begin{lem}
The map $f:\Gamma\to \P(W/V)$ in \eqref{eq:fGamma} is constant on $\Gamma_-\Gamma_0\Gamma^\P_V$ 
cosets, and therefore it descends to $\caJ_V = \Gamma/(\Gamma_-\Gamma_0\Gamma^\P_V)$.
\end{lem}
\begin{proof}
It suffices to check the invariance of $f$ along $\Gamma_0\Gamma^\P_+$ cosets. Let
$\gamma'\in\Gamma$ and $\gamma\in\Gamma^\P_+$, then by \eqref{eq:flow1} we have
\[
\psi_W(\gamma'\gamma^{-1}) = \psi_{W\gamma}(\gamma')\gamma^{-1}.
\]
Therefore, modulo $V$ this is independent of $\gamma$. Similarly, by \eqref{eq:flow1} and the proof of
Theorem \ref{thm:periodic}(a),
\begin{equation}\label{eq:psit_0}
\psi_W(\gamma'e^{-it_0}) = e^{-it_0}\psi_W(\gamma')
\end{equation}
and therefore the projective map $f$ is constant along $\Gamma_0$-orbits.
\end{proof}
\begin{cor}\label{cor:sufficient}
Let $\ell_V:\C\to\caJ_V$ be the homomorphism $($of real groups$)$ obtained by the 
composition $\C\simeq\Gamma_1\to\Gamma\to\caJ_V$. Then $f:\C\to\P(W/V)$ factors through this.
Hence a sufficient condition for \eqref{eq:f} to be $\Lambda$-periodic is that $\ell_V$ be 
$\Lambda$-periodic.
\end{cor}

\subsection{A construction for solutions of finite type.}\label{subsec:finite} 

Let us now consider a more concrete version of the above construction which
applies to solutions of finite type. The spectral data 
$(X,\rho,P,\zeta,\caL,\varphi)$ provides us with $W$; we can choose a
$\H$-codimension $n+1$ $\H$-subspace $W_\fq$ of $W$ as follows. Fix a $\rho$-invariant
divisor $\fq$ consisting of $2n+2$ distinct smooth points in $X\setminus\{P,\rho P\}$.
We will write
\[
\fq=Q_0+\cdots+Q_n+\rho(Q_0)+\cdots+\rho(Q_n).
\]
We may choose the local parameter disc $\Delta_P$ so that $\fq\subset X_0$. Let $W_\fq$
correspond to the subspace $H^0(X_0,\caL(-\fq))$ of holomorphic
sections of $\caL$ which vanish on $\fq$: it is a left $\H$-subspace since $\fq$ is
$\rho$-invariant. Since $X_0$ is a Stein manifold it
follows, by calculating the cohomology of the sheaf exact sequence
\begin{equation}\label{eq:Lq}
 0\to \caL(-\fq)\to\caL\to\caL_\fq\to 0,
\end{equation} 
that $\dim_\C(W/W_\fq)=\deg(\fq)=2n+2$.


Now consider the dependence of $f$ on the choice of trivialisation
$\varphi$. By Remark \ref{rem:change_of_triv} any change of trivialisation
has the form $\varphi\mapsto\varphi e^{s_0+it_0}\gamma'$ where
$\gamma'\in\Gamma_-$. By Remark \ref{rem:gamma} this only alters $f$ by 
congruence. Further, a change of local parameter $\zeta$ amounts to rescaling the parameter $z$ 
in the Dirac operator. Therefore it is enough to be given $(X,\rho,P,\caL,\fq)$ to have $f$
well-defined up to congruence in $\HP^n$: we may as well take $\psi_W$ to be the 
global Baker function.
The condition above that $f$ be well-defined is that $\psi_W$ does not vanish on $\fq$ for any
$z\in\C$. We can think of $W_\fq$ as the kernel of $n+1$ left $\H$-linear forms
$\ev_{Q_0},\ldots,\ev_{Q_n}$ on $W$, where $\ev_Q$ is obtained by extending the map ``evaluation
at $Q$'' to all of $W$ (it is straightforward to check that $\ev_{\rho Q}=\ev_Qu$ for some
$u\in\H^\times$). In this case, in homogeneous coordinates, and up to congruence in $\HP^n$,
\begin{equation}\label{eq:fcoords}
f(z) = [\psi_W(z,Q_0),\ldots,\psi_W(z,Q_n)].
\end{equation}
Further, for any $\H$-codimension two subspace $V$ satisfying $W_\fq\subset V\subset W$ we obtain a
projection of $f$ onto $\P(W/V)\simeq\HP^1$, which will be a conformal immersion provided $V$ is
well-positioned in $W$. Equally, this arises from a well-positioned two $\H$-dimensional linear system
in the dual space $(W/W_\fq)^*$.

To study the periodicity conditions for $f$ when it arises from
this construction we begin by showing that when $V=W_\fq$ the group $\caJ_V$ appearing in Lemma
\ref{cor:sufficient} is isomorphic to a real subgroup of a generalised Jacobi variety.
Let $X_\fq$ be the singularization of $X$ obtained by identifying the points of $\fq$ together
simply. 
The real involution $\rho$ descends to $X_\fq$, so its generalised Jacobian $\Jac(X_\fq)$ has
a real subgroup $J_\R(X_\fq)$ which is the connected component of the identity
of the real subgroup of $\bar\rho^*$-fixed line bundles over $X_\fq$. 
Let $X_\fq^\sm$ denote the curve of smooth points on $X_\fq$ (this is just $X^\sm\setminus\fq$)
and let $\caA_P^\fq:X_\fq^\sm\to \Jac(X_\fq)$ denote the Abel map
for $X_\fq$, with base point $P$. Let $\ell^\fq:\C\to J_\R(X_\fq)$ 
be the unique homomorphism of real groups determined by the property
\begin{equation}\label{eq:tangency}
\frac{\partial \ell^\fq}{\partial z}|_{z=0} = 
\frac{\partial\caA^\fq_P}{\partial\zeta^{-1}}|_{\zeta^{-1}=0}.
\end{equation}
\begin{lem}\label{lem:pair}
For a pair $(W,W_\fq)$ given by spectral data $(X,\rho,P,\caL,\fq)$ in the manner above, 
we have $\caJ_{W_\fq}\simeq J_\R(X_\fq)$ and $\ell_{W_\fq} = \ell^\fq$.
\end{lem}
\begin{proof}
First we note that an equivalent way of describing $\ell^\fq$ above is that the line bundle
$\ell^\fq(z)$ is obtained using transition functions $e^{z\zeta}$ and $e^{\bar
z\overline{\rho^*\zeta}}$ to glue together trivial bundles as in \eqref{eq:transition}, but where
$X\setminus\{P,\rho P\}$ is replaced by its singularisation $X_\fq\setminus\{P,\rho P\}$. This
extends naturally to a real homomorphism $\ell^\fq:\Gamma\to J_\R(X_\fq)$ 
by using $\gamma$ and $\overline{\rho^*\gamma}$ as transition functions. 

Now $\gamma$ lies in $\ker(\ell^\fq)$ precisely when $\gamma$ factorises into a product
$\gamma=\alpha\beta$ where $\alpha$ extends holomorphically to $\Delta_P$ and $\beta$ extends
holomorphically to $(X_\fq)_0$, i.e., $\beta$ represents the boundary of a holomorphic function on
$X_0$ which takes the same value at each point of $\fq$. 
Clearly $\Gamma_-\Gamma_0$ equals the group of all boundaries of the type $\alpha$, while 
the boundaries of the type $\beta$ are exactly those which, by multiplication, fix
\[
H^0(X_0,\caL)/H^0(X_0,\caL(-\fq))\simeq W/W_\fq
\]
projectively, since they act by scaling. Hence $\ker(\ell^\fq) =
\Gamma_-\Gamma_0\Gamma^\P_{W_\fq}$, whence the result.
\end{proof}
\begin{rem}\label{rem:Jac(Xq)}
One knows (from e.g., \cite{Ser}) that $\Jac(X_\fq)$ is a group extension of $\Jac(X)$:
\begin{equation}\label{eq:extension}
1\to K\to \Jac(X_\fq)\stackrel{\pi_\fq}{\to} \Jac(X)\to 1,
\end{equation}
where $K$ is a linear algebraic group. In our case $K\simeq (\C^\times)^{2n+2}/\C^\times$.
The real automorphism
$\bar\rho^*$ acts on $\Jac(X_\fq)$ and preserves $K$. Define $K_\R= K\cap J_\R(X_\fq)$,
then
\[
K_\R\simeq (\C^\times)^{n+1}/\R^\times.
\]
This is the kernel of the restriction of $\pi_\fq$ to $J_\R(X_\fq)$.
\end{rem}
We can now achieve a more precise understanding of how the map $f$ factors through $\Jac(X_\fq)$ by
considering the natural twistor lift  which every \qh curve
$f:M\to\HP^n$ possesses. This is a map $\hat f:M\to\CP^{2n+1}$ for which
$T\circ\hat f = f$, where $T:\CP^{2n+1}\to\HP^n$ is the twistor projection assigning to every
left $\C$-line in $\H^{n+1}$ its left $\H$-line. The lift arises as follows: the corresponding line
bundle $L\subset\underline{\H}^{n+1}$ possesses a unique complex structure $J$ for which 
$*\delta f = \delta f J$. The twistor lift $\hat f$ is given by the complex line subbundle $\hat
L\subset L$ of $i$-eigenspaces 
of $J$. Now let $\tilde f:\C\to\H^{n+1}$ represent $\psi_W+W_\fq$ in some choice of
coordinates on $W/W_\fq$ and write $f = \H\tilde f$. It
follows from the proof of Theorem \ref{thm:f} that $\hat f = \C\tilde f$. 
\begin{thm}\label{thm:composite}
Given spectral data as above, the natural twistor lift of $f:\C\to\HP^n$ 
is a composite of the form
\begin{equation}\label{eq:composite}
\hat f:\C\stackrel{\ell^{\fq}}{\to}\Jac(X_\fq)\stackrel{\theta}{\to}\CP^{2n+1},
\end{equation}
where $\theta$ is a rational map.
\end{thm}
\begin{proof}
We will show that $\hat f$ comes from a construction similar to the one given in \cite{McI01} for
harmonic tori. To this end, let $\caE\to\Jac(X)$ denote the complex rank $2n+2$ vector bundle with
fibres
$\caE_L = H^0(X,(\caL\otimes L)_\fq)$. 
It was shown in \cite{McI01} that we can embed $\Jac(X_\fq)$ into the bundle of
(complex) projective frames of $\caE$ and therefore the tautological section of 
the pullback $\pi_\fq^*\Jac(X_\fq)$ of $\Jac(X_\fq)$ over itself globally (and
algebraically) trivialises bundle $\P\caE'$ of complex projective space of $\caE'=\pi_\fq^*\caE$, 
by canonically identifying each fibre of $\P\caE'$ with $\P H^0(X,\caL_\fq)\simeq \CP^{2n+1}$. 
This works as follows: each point of $\Jac(X_\fq)$
should be thought of as a line bundle $L\in\Jac(X)$ equipped with a trivialising section of $L_\fq$
determined up to scaling. This fixes a projective identification of
$(\caL\otimes L)_\fq$ with $\caL_\fq$ by tensor product. We will denote this canonical 
trivialisation by
\begin{equation}\label{eq:tau}
\tau:\P\caE'\to\Jac(X_\fq)\times \P H^0(X,\caL_\fq).
\end{equation}
Now define
\[
\caU = \{L'\in\Jac(X_\fq);\caL(-P-\rho P)\otimes \pi_\fq(L')\ \text{is non-special}\}.
\]
This is an affine open subvariety of $\Jac(X_\fq)$ and it is a consequence of Theorem \ref{thm:bigcell}
that $J_\R(X_\fq)\cap\caU$ is the complement of a real analytic subvariety of $J_\R(X_\fq)$. 
For each $L\in \pi_\fq(\caU)$ the vector space $H^0(X,\caL(-\rho P)\otimes L)$ is one-dimensional and
therefore the bundle $\P\caE'$ possesses an algebraic section $s:\caU\to\P\caE'$ corresponding to the
injection
\begin{equation}\label{eq:twistor}
H^0(X,\caL(-\rho P)\otimes L)\to H^0(X,\caL\otimes L)\to H^0(X,(\caL\otimes L)_\fq).
\end{equation}
Now we define $\theta:\Jac(X_\fq)\to\CP^{2n+1}$ by $\theta = \tau\circ s$, having fixed some projective
linear identification of $ \P H^0(X,\caL_\fq)$ with $\CP^{2n+1}$. This is algebraic on the
affine open subvariety $\caU$ and therefore rational. The theorem is proved once we have shown
that, via the identification 
\[
H^0(X,\caL_\fq) = H^0(X_0,\caL)/H^0(X_0,\caL(-\fq))\stackrel{w}{\to} W/W_\fq,
\]
derived from \eqref{eq:w}, the map $\hat f$ corresponds to $\theta\circ\ell^\fq$. To see this we
note that $\ell^\fq(z)$ is defined using the transition relation \eqref{eq:transition} to glue
together the trivial bundles over $X\setminus X_0$ and $(X_\fq)_0$, and therefore can be thought of
as $\ell(z)$ equipped with the trivialisation $\tau_0(z)$ restricted to $\fq$. But now we recall
that the global Baker function in \eqref{eq:global} is the result of applying \eqref{eq:w} to the
global section $\sigma_\gamma\in H^0(X,\caL(-\rho P)\otimes\ell(\gamma))$ pulled back to a section
of $\caL$ over $X_0$ by $\sigma_z\to\sigma_z\tau_0(z)^{-1}$. When we restrict this to
$\caL_\fq$ we see that this generates the line $\tau\circ s\circ\ell^\fq(z)$ in $H^0(X,\caL_\fq)$.
Hence $\hat f = \theta\circ\ell^\fq$. 
\end{proof}
\begin{rem}\label{rem:T}
From the previous proof we see that $f=T\circ\hat f$ is directly obtained from the $\H$-line 
subbundle of $\caE'$ whose fibres are the left $\H$-lines 
$H^0(X,\caL\otimes L)\subset H^0(X,(\caL\otimes L)_\fq)$. Notice that if $W_\fq\subset V\subset W$
for a well-positioned $\H$-codimension two subspace $V$ and $f_V:\C\to\HP^1$ is the corresponding 
projection of $f$ then the twistor lift of $f_V$ is the projection of $\hat f$ onto $\CP^3$, i.e.,
the map 
\[
\hat f_V:\C\to\P_\C(W/V);\quad z\mapsto \C.(\tilde f(z)+V).
\]
\end{rem}
Now we return to the question of when the map $f$ is doubly periodic. 
\begin{thm}\label{thm:fperiodic}
Let $f:\C\to\HP^n$ correspond to the spectral data above and suppose that it is linearly full
$($i.e., its image does not lie in some $\HP^{n-1})$.
Then $f$ and $\ell:\C\to\Jac(X)$ are simultaneously $\Lambda$-periodic if 
and only if $\ell^\fq$ is $\Lambda$-periodic.
\end{thm}
This implies that the monodromy of the corresponding Dirac potential \eqref{eq:monodromy} 
is that of the flat $S^1$-bundle $\caS$ over $\C/\Lambda$ described in Remark \ref{rem:S}.
\begin{proof}
That the $\Lambda$-periodicity of $\ell^\fq$ is sufficient follows at once from Corollary
\ref{cor:sufficient} and Lemma \ref{lem:pair}. Now suppose both $f$ and $\ell$ are $\Lambda$-periodic. Then $\hat f$ is
$\Lambda$-periodic, and $\ell^\fq$ determines a homomorphism
\[
h:\Lambda\to K_\R\subset\ker(\pi_\fq);\quad h(\lambda) = \ell^\fq(\lambda).
\]
In this situation we can use \cite[Lemma 1]{McI01} to deduce that there is an injective
homomorphism of $K_\R$ into $\P GL(n+1,\H)$ which allows us to write 
$f(z+\lambda) = f(z)h(\lambda)$ for all $\lambda\in\Lambda$. Now since $f$ is linearly full
we can, after possibly a congruence, find homogeneous coordinates for $f(z)$ so that these 
are all non-zero at some $z$, hence $h$ must be the identity, whence $\ell^\fq$ is $\Lambda$-periodic.
\end{proof}
\begin{rem}\label{rem:fperiod}
I was unable to find a way to remove the assumption that $\ell$ be $\Lambda$-periodic: it
amounts to the difference between knowing that the Dirac potential has $\Lambda$-monodromy
\eqref{eq:monodromy} (which follows from the periodicity of $f$) and knowing that the full 
QKP operator $L_W$ has $\Lambda$-monodromy (i.e., $[L_W]$ is $\Lambda$-periodic). The former
implies the latter if one knows \emph{a priori} that the QKP Baker function has property (c) of Theorem
\eqref{thm:spectrum}. We will see later in the article that under the
assumption that $\ell$ is $\Lambda$-periodic we do get agreement of the two Baker functions. 
\end{rem}
\begin{rem}\label{rem:g=0}
In the case where $X$ has genus $0$ or $1$ the map $\ell:\C\to\Jac(X)$ is not injective, so the
periodicity condition on $\ell$ does not uniquely determine a lattice $\Lambda\subset\C$. But the
arithmetic genus of $X_\fq$ is at least $2n+1\geq 3$ hence $\ell_\fq$ is always injective,
therefore it uniquely determines the lattice $\Lambda$ when such a lattice exists. 
\end{rem}
Finally, let us note a condition under which the conformal torus in $\HP^1$ can
be immersed in $\R^3$, at least in the case where $X$ is a smooth curve. Recall from the discussion 
at the end of \S\ref{subsec:Dirac} that $f:M\to
S^4$ lies in some $\R^3$ if and only if $\hat E$ is a spin bundle, hence if and only if $\caS$ is a
spin bundle. 
\begin{prop}\label{prop:hyperelliptic}
When $X$ is a smooth curve the bundle $\caS$ above is a spin bundle if the divisor 
$P-\rho P$ has order two $($in which case $X$ is hyperelliptic$)$.
\end{prop}
\begin{proof}
One knows (e.g., from \cite{Ser}) that pulling back line bundles along the Abel map gives 
an isomorphism between $\Jac(X)$ and its dual
$\Jac(X)^*$ (the moduli space of flat line bundles on $\Jac(X)$). One also knows from 
\cite[VII]{Ser} that
the flat line bundle over $\Jac(X)$ with principal bundle $\Jac(X_\fp)$ is pulled back to 
$\caO_X(P-\rho P)$.  Hence $\caS$ is a spin bundle whenever $\caO_X(2P-2\rho
P)$ is trivial, i.e., when $P-\rho P$ is a divisor of order two.
\end{proof}
It can be shown that this condition obliges the ``even'' flows of the QKP hierarchy to be
trivial, in which case we obtain the modified Novikov-Veselov hierarchy, as one expects
from \cite{Tai04}.

\section{Darboux transformations.}

We will follow the notion of Darboux transformations introduced by Bohle et
al.\ in \cite{BohLPP}. It generalises the classical notion of a Darboux 
transform
between isothermic surfaces in $S^3$, in which both surfaces envelope the same
sphere congruence.

Let us recall first, from \cite{BurFLPP}, that we can identify the set of all
oriented round 2-spheres in $S^4$ with the set 
\[
\caZ=\{S\in\End_\H(\H^2); S^2=-I\}.
\]
This identification gives to each $S\in\caZ$ the 2-sphere 
$\{L\in\HP^1;SL=L\}$, which we will also denote by $S$. The orientation is
given by the complex structure each 2-sphere inherits from $S$.
Given a Riemann surface, $M$, a sphere congruence is a map $S:M\to\caZ$. Now
let $f:M\to S^4$ be a conformal map. We say $f$ envelopes
a sphere congruence $S:M\to\caZ$ if $f(p)\in S(p)$ and the
oriented tangent plane of $f(M)$ at $f(p)$ agrees with that of $S(p)$. In 
terms of the line subbundle $L\in \HH^2$ corresponding to $f$,
these two conditions can be expressed as
\[
SL=L,\quad\text{and}\quad *\delta f = S\delta f= \delta fS.
\]
It is a classical result that 
two conformal maps which envelope the same sphere congruence
must both be isothermic (see \cite[Cor.\ 67, p.78]{Boh} for a modern proof), 
therefore Bohle et al.\ \cite{BohLPP} (see also \cite{Boh}) relax the condition
slightly to achieve a broader class of transformations. We will say a
conformal map $f:M\to S^4$ left-envelopes a sphere congruence $S:M\to\caZ$ if
$f(p)\in S(p)$ and their oriented tangent planes agree up to action of
$SU_2$ on $T_{f(p)}S^4$ as the left component in the epimorphism
$SU_2\times SU_2\to SO_4$ (the notion of a right-envelope is defined
similarly). In terms of the line bundle $L$ the property of being a
left-envelope can be expressed as
\[
SL=L,\quad\text{and}\quad *\delta f =S \delta f.
\]
\begin{defn}[\cite{BohLPP}]
Let $f:M\to S^4$ be a conformal map of a Riemann surface $M$. Another
conformal map $f^\sharp:M\to S^4$ is a Darboux transform of $f$ if $f(p)\neq
f^\sharp(p)$ for all $p\in M$ and there
exists a sphere congruence $S:M\to\caZ$ which is enveloped by $f$ and
left-enveloped by $f^\sharp$. In terms of the line bundles $L,L^\sharp$ this
means
\begin{equation}
\HH^2=L\oplus L^\sharp,\ SL=L,\ SL^\sharp=L^\sharp,\ 
*\delta f = S\delta f = \delta f S,\ \text{and}\ *\delta f^\sharp = S \delta f^\sharp .
\end{equation}
\end{defn}
Strictly speaking, we want to allow \emph{singular} Darboux transformations of a torus: 
those for which $L\cap L^\sharp$ is trivial except at finitely many points (see \cite{BohLPP}).
We want to understand what Darboux transformations look like for maps 
arising from a pair $(W,V)$ of the type above. First we invoke a simple result
from \cite{BohLPP} which gives a neat characterization for Darboux transforms.
\begin{lem}[\cite{BohLPP}]
Let $f,f^\sharp:M\to \HP^1$, with corresponding line bundles $L,L^\sharp$,
be conformal immersions so that $\HH^2 = L\oplus
L^\sharp$. Then $f^\sharp$ is a Darboux transform for $f$ if and only if
$*\delta f^\sharp = J\delta f^\sharp$ where $*\delta f= \delta f J$ and we
identify $\HH^2/L^\sharp$ with $L$ using projection along the
splitting.
\end{lem}
It is easy to check that the sphere congruence which is enveloped by $f$ and
left-enveloped by $f^\sharp$ is given by $S|_L=J$ and $S|_{L^\sharp} = \tilde
J$ where $*\delta f = \tilde J \delta f$, again using the splitting to
identify $L^\sharp$ with $\HH^2/L$.
\begin{lem}\label{lem:DT}
Let $f:M\to\P(W/V)\simeq\HP^1$ be the conformal map defined by equation
\eqref{eq:f}. Given $\lambda\in\C$, $0<|\lambda|<1$, define $W_\lambda = W(1-\lambda\zeta)$, 
$V_\lambda = V(1-\lambda\zeta)$ and let $f^\lambda:M'\to\HP^1$ be the map corresponding to
the pair $(W_\lambda,V_\lambda)$. Then $f^\lambda$ is a Darboux transform of $f$
over their common domain $M\cap M'$.
\end{lem}
\begin{proof}
We assume, without loss of generality, that $M=M'$ is an open domain in $\C$. 
Let $\caV = W_\lambda/V_\lambda$ and let
$f,f^\lambda:M\to\P\caV$ denote the conformal maps with lifts
\[
\tilde f = \psi_W(1-\lambda\zeta) + V_\lambda,\quad
\tilde f^\lambda = \psi_{W_\lambda} + V_\lambda,
\]
and line subbundles $L,L^\lambda\subset M\times\caV$. 
From the proof of Theorem \ref{thm:f} we know that $*\delta f = \delta f J$ for
$J\tilde f = i\tilde f$. 
I claim that there is a function $b:M\to\H$ for which
\begin{equation}\label{eq:tildef}
\tilde f^\lambda_z = b\tilde f^\lambda -\frac{1}{\lambda} \tilde f.
\end{equation}
Therefore
\[
\delta f^\lambda (\tilde f^\lambda) = -dz\frac{1}{\lambda} \tilde f.
\]
It follows that, since $\lambda$ is complex,
$*\delta f^\lambda = J\delta f^\lambda$. In light of the previous lemma, this
proves the theorem. 

It remains to verify equation \eqref{eq:tildef}. This is a direct result
of an identity for Baker functions. Set
\[
\psi = \psi_W(1-\lambda\zeta),\quad \psi^\lambda = \psi_{W_\lambda}.
\]
These have respective Fourier series expansions
\begin{eqnarray*}
\psi = (-\lambda\zeta + (1-a\lambda) +\cdots)e^{z\zeta}\\
\psi^\lambda = (1 + a^\lambda\zeta^{-1} +\cdots)e^{z\zeta}.
\end{eqnarray*}
Therefore there is an $\H$-valued function $b$ for which
\[
(\psi^\lambda_z +\frac{1}{\lambda}\psi -b\psi^\lambda)e^{-z\zeta}\equiv 0\bmod
O(\zeta^{-1}).
\]
Since $W_\lambda e^{-z\zeta}$ is transverse to $H_-$ almost 
everywhere (Theorem \ref{thm:bigcell}), the right hand side must be identically zero.
Equation \eqref{eq:tildef} follows.
\end{proof}
Notice that $(1-\lambda\zeta)\in\Gamma_+$, since its only zero is at
$\lambda^{-1}$. It is clear that the proof above still works if we replace $(1-\lambda\zeta)$ by
any representative in the coset $(1-\lambda\zeta)\Gamma_-\in\Gamma/\Gamma_-$. It follows that
in the case where $W_\fq\subset V\subset W$ for some $\rho$-invariant divisor $\fq\subset
X\setminus\{P,\rho P\}$ this Darboux transform acts on the spectral data
by fixing everything except the pair $\caL,\varphi$, which it transforms by
\begin{equation}\label{eq:DT}
(\caL,\varphi)\mapsto (\caL^Q,\varphi^Q),\quad \caL^Q=\caL(Q+\rho Q -P-\rho P),
\end{equation}
where $\zeta(Q)=\lambda^{-1}$ and $\varphi^Q$ is a local trivialising section of $\caL^Q$ over
$\Delta_P$.  In fact, provided $Q$ is not in the support of
$\fq$, the transform \eqref{eq:DT} always yields a Darboux transform.
\begin{thm}\label{thm:darboux}
Let $f:\C/\Lambda\to\HP^1$ be a quaternionic holomorphic
torus corresponding to the spectral data $(X,\rho,P,\caL,\fq)$ and a choice of well-positioned two
dimensional linear system $H\subset (W/W_\fq)^*$. Assume that $\ell:\C\to\Jac(X)$
is $\Lambda$-periodic.  Let $X_\fq^\sm$ be the variety of smooth points in $X_\fq$.
Then for every $Q\in X_\fq^\sm\setminus\{P,\rho P\}$ there is a Darboux transform 
$f^Q$ of $f$ arising from the spectral data obtained from the transformation \eqref{eq:DT}.
\end{thm}
The linear system $H$ is spanned by two linear forms each of which is a $\H$-linear
combination of evaluation maps $\ev_{Q_0},\ldots,\ev_{Q_n}$. Therefore $f^Q$ corresponds to the
linear system $H^Q\subset (W/W_\fq)^*$ determined by the same combination of evaluations.
Since both $\ell$ and $f$ are assumed $\Lambda$-periodic Theorem \ref{thm:fperiodic} ensures that
$\ell^\fq$ is $\Lambda$-periodic, and therefore the $\Gamma_1$-orbit of $\caL^Q$ in $\Pic(X_\fq)$ is 
isomorphic to $\C/\Lambda$.
The proof now follows from Lemma \ref{lem:DT} by choosing a coordinate
disc on $X$ centred at $P$ and containing $Q$ but not any points in $\fq$.
The singularities of such a Darboux transform  correspond to points where the $\Gamma_1$-orbit
of $W^Q$, the point in $Gr_\QKP$ corresponding to the transformed spectral data, leaves the big cell. 

\section{Spectral curves.}\label{sec:spectral}

\subsection{The QKP spectral curve.}\label{subsec:QKPcurve}
In the construction of $W\in Gr_\QKP$ from spectral data we
use the map $w$ in \eqref{eq:w}. When this is restricted to the algebraic sections of $\caL$ over
$X_0$ (i.e., those which have only poles at $P,\rho P$) we obtain an open dense subspace
$W^\alg\subset W$. The elements of $W^\alg$ are algebraic in the sense that their projections onto
$H_+$ are polynomial in $\zeta$. Now if $\caA$ denotes the coordinate ring of $X\setminus\{P,\rho
P\}$ its real subalgebra
\[
\caA^\rho = \{h\in\caA;\overline{\rho^*h}=h\},
\]
acts on $W^\alg$ by right multiplication, since $w(h\sigma) = w(\sigma)h$ (where on the right hand
side we restrict $h$ to the circle $C_P$). Therefore
\[
\caA^\rho\hookrightarrow \caA_W = \{h\in C^\omega(S^1,\C);W^\alg h\subset W^\alg\}.
\]
In fact it is easily shown that this inclusion is onto, so that $\caA^\rho\simeq\caA_W$. Thus we
recover $X\setminus\{P,\rho P\}$ as $\Spec(\caA^\C_W)$, where $\caA^\C_W$ is the complex subalgebra
of $C^\omega(S^1,\C)$ generated by $\caA_W$.
Now $(W^\alg)^\C$, the complexification $W^\alg+W^\alg i\subset H$, is a torsion free 
$\caA^\C_W$-module and this recovers the rank one torsion
free coherent sheaf $\caL$ equipped with the trivialisation implicit in the inclusion
$(W^\alg)^\C\subset H^\alg_+\oplus H_-$. 

Now, in analogy with the KP case, it is clear that we can assign a commutative algebra $\caA_W$
to \emph{any} $W\in Gr_\QKP$, but in general this will not be very useful: typically $\caA_W=\R$,
and even if it is not so trivial we only obtain spectral data of the type we desire when
$(W^\alg)^\C$ is locally rank one. Nevertheless, we can always obtain $\caA^\rho$ as a commutative
algebra of differential operators over $\H$, and this connects us to the stationary QKP flows.
The proof of the following lemma is obtained \textit{mutatis mutandis} from 
\cite[Remark 6.4]{SegW}.
\begin{lem}
Given $W\in Gr_\QKP$ to each $h\in\caA_W$ there is a unique pseudo-differential operator 
$P(h)\in Z_0(L_W)$ for which the differential operator part $P(h)_+$ satisfies 
$P(h)_+\psi_W = \psi_Wh$. The $\R$-algebra $\{P(h)_+;h\in\caA_W\}$ is 
isomorphic to $\caA_W$. It follows that $[P(h)_+,L_W]=0$ and therefore the QKP solution 
corresponding to $W$ is stationary for every flow $\partial_{P(h)}$.
\end{lem}
One consequence of this lemma is that $L_W$ admits only finitely many independent non-stationary
QKP flows
precisely when the algebra $\caA_W$ possesses an element of every sufficiently high order, and
therefore $(W^\alg)^\C$ is locally rank one. Hence $L_W$ is a solution of finite type.

\subsection{The curve of Darboux transforms.} 
Assume we have a \qh torus $f:\C/\Lambda\to\HP^n$ obtained from spectral 
data $(X,\rho,P,\caL,\fq)$ as described above. According to Bohle et al.\ \cite{BohLPP} there will be
a holomorphic curve in $\CP^3$ given by the Darboux transforms of any torus in $\HP^1$ obtained from
$f$ by projection.
We are going to use the interpretation of the Darboux transform given in \eqref{eq:DT} to work
directly with $f$ and view the
curve of Darboux transforms as an algebraic curve in $\CP^{2n+1}$.
For this purpose, let $f^Q$ be the \qh torus obtained from the
transformation \eqref{eq:DT}. Any projection of $f$ onto $\HP^1$ then has a Darboux transform by
applying the same projection to $f^Q$ (see Remark \ref{rem:T}).

Since $f^Q=f^{\rho Q}$, a direct consequence of 
Theorem \ref{thm:darboux} is a geometric realisation of the Klein surface $X_\fq^\sm/\rho$
(recall that $X_\fq^\sm\subset X\backslash\fq$ is the subvariety of smooth points), 
since we obtain from it a map
\begin{equation}\label{eq:F}
F: (\C/\Lambda)\times (X_\fq^\sm/\rho) \to\HP^n;\ (p,Q+\rho Q)\mapsto f^Q(p).
\end{equation}
Here we define $f^P=f$. An immediate consequence of Theorems
\ref{thm:composite} and \ref{thm:darboux} is that this factors through the generalised Jacobian
$\Jac(X_\fq)$ via
\[
(\C/\Lambda)\times (X_\fq^\sm/\rho)\to\Jac(X_\fq);\ 
(z,Q+\rho Q)\mapsto \ell^\fq(z)\otimes\caO_{X_\fq}(Q+\rho Q-P-\rho P).
\]
Post-composition of this with $T\circ\theta:\Jac(X_\fq)\to\HP^n$ (cf.\ Remark \ref{rem:T}) 
gives us $F$. 
For each $p\in\C/\Lambda$ let us define 
\[
\xi_p:X_\fq^\sm/\rho\to\HP^n;\quad \xi_p(Q) = F(p,Q)=f^Q(p).
\]
This has a natural twistor lift, an algebraic map from $X_\fq^\sm$ into $\CP^{2n+1}$, which factors
through the Abel map into the generalised Jacobian $\Jac(X_\fq)$, as a consequence of Theorem 
\ref{thm:composite}. 

Let $S^2X_\fq^\sm$ denote the symmetric product of $X_\fq^\sm$ with itself; equally, think of it 
as the set of
all divisors of degree two supported on $X_\fq^\sm$. Because $X_\fq^\sm$ excludes $\fq$ and all 
singularities of $X$, this symmetric product admits an Abel map
\begin{equation}\label{eq:Abel}
\caA_{P+\rho P}^\fq:S^2X_\fq^\sm\to\Jac(X_\fq);\quad A+B\mapsto \caO_{X_\fq}(A+B-P-\rho P).
\end{equation}
We can embed $X_\fq^\sm$ algebraically in $S^2X_\fq^\sm$ via $Q\mapsto Q+\rho P$; we can also embed
$X_\fq^\sm/\rho$
real algebraically by thinking of it as the curve of pairs $Q+\rho Q$ for $Q\in X_\fq^\sm$. By
post-composing each of these with the Abel map \eqref{eq:Abel} we obtain
\begin{eqnarray*}
\alpha:X_\fq^\sm\to\Jac(X_\fq);&\quad Q\mapsto\caO_{X_\fq}(Q-P),\\
\beta:X_\fq^\sm/\rho\to\Jac(X_\fq);&\quad Q+\rho Q\mapsto \caO_{X_\fq}(Q+\rho Q -P-\rho P).
\end{eqnarray*}
From the discussion above we see that, if $\caL$ corresponds to the base point $p\in\C/\Lambda$, 
\[
\xi_p = T\circ\theta\circ\beta.
\]
We define
\begin{equation}\label{eq:hatxi}
\hat\xi_p=\theta\circ\alpha:X_\fq^\sm\to\CP^{2n+1}.
\end{equation}
The image of $\hat\xi_p$ will be called the \emph{Darboux spectral curve}. It is clearly an
algebraic curve.
\begin{thm}\label{thm:Darbouxcurve}
$\hat\xi_p$ is a twistor lift of $\xi_p$.
\end{thm}
This can be summarised by the following commutative diagram, in which the top line is the algebraic
map $\hat\xi_p$ and the bottom line is $\xi_p$.
\begin{equation}
\begin{array}{cccccc}
X_\fq^\sm & \stackrel{\alpha}{\rightarrow}&\Jac(X_\fq) & \stackrel{\theta}{\rightarrow}&\CP^{2n+1}\\
\downarrow & & & & \downarrow{T}\\
X_\fq^\sm/\rho&\stackrel{\beta}{\rightarrow}&J_\R(X_\fq)&\stackrel{T\circ\theta}{\rightarrow} & \HP^n
\end{array}
\end{equation}
\begin{proof}
First we note that 
$\caO_{X_\fq}(A+B-P-\rho P)$ can be thought of as the bundle $\caO_X(A+B-P-\rho P)$ together with
the fibre identification over $\fq$ uniquely determined by the rational section with divisor
$A+B-P-\rho P$ (there is only one of these, up to scaling). Consequently the canonical trivialisation 
of the complex projective bundle $\P\caE'$
over $\Jac(X_\fq)$ (described in the proof of Theorem \ref{thm:composite}) works as follows over
$\alpha(X_\fq^\sm)$. Let $\sigma_Q$ be a non-zero rational section of $\caO_X(Q-P)$ with divisor $Q-P$.
Recall from \eqref{eq:Lq} that $\caL_\fq$ denotes the skyscraper sheaf which is $\caL$ restricted to
$\fq$. The canonical trivialisation of $\P\caE'$ identifies fibres over $\alpha(X_\fq^\sm)$ via
the isomorphism 
\[
\iota_1:H^0(X,\caL(Q-P)_\fq)\to H^0(X,\caL_\fq);\quad s\mapsto s/\sigma_Q,
\]
thinking of $\sigma_Q$ restricted to $\fq$. Similarly, over $\beta(X_\fq^\sm)$ we have an
isomorphism
\[
\iota_2:H^0(X,\caL^Q_\fq)\to H^0(X,\caL_\fq);\quad s\mapsto s/(\sigma_Q{\bar\rho^*\sigma_Q}).
\]
Now recall from \eqref{eq:twistor} that $\theta$ is the result of applying this trivialisation to 
the line subbundle of $\caE'$ which picks out the complex line
\begin{equation}\label{eq:theta}
H^0(X,\caL(-\rho P)\otimes L)\subset H^0(X,(\caL\otimes L)_\fq)=\caE'_{L'},\quad L=\pi_\fq(L'),
\end{equation}
while its twistor projection $T\circ\theta$ corresponds to the quaternionic line
\[
H^0(X,\caL\otimes L)\subset H^0(X,(\caL\otimes L)_\fq).
\]
The statement of the theorem is that
\begin{equation}\label{eq:alphabeta}
T\circ\theta\circ\alpha(Q) = T\circ\theta\circ\beta(Q+\rho Q). 
\end{equation}
To prove this, consider 
\begin{equation*}
\iota_2^{-1}\circ\iota_1(s)= s{\bar\rho^*\sigma_Q},
\end{equation*}
when $s$ is a non-zero globally holomorphic section of $\caL(Q-P)$ with a zero at
$\rho P$. 
The result is a globally holomorphic section of $\caL^Q$, since the simple pole of
$\bar\rho^*\sigma_Q$ at $\rho P$ is cancelled by the zero of $s$. Thus $\iota_1(s)$ lies in the
$\H$-subspace of $H^0(X,\caL_\fq)$ given by the image of $H^0(X,\caL^Q)$ under $\iota_2$.
This proves \eqref{eq:alphabeta}.
\end{proof}

\subsection{The multiplier spectrum.}
Here we will examine how the multiplier spectrum $\Sp(L^*,D)$ is related to the
spectral curve $X$ in the case of a \qh curve $L$ arising from
spectral data in the manner of Section \ref{subsec:finite}. 

Suppose that we have a conformally immersed torus $f:\C/\Lambda\to\HP^1$ of finite type,
with spectral data $(X,\rho,P,\caL,\fq)$, i.e., we suppose
that the map $\ell^\fq:\C\to J_\R(X)$ is $\Lambda$-periodic. We may as well assume $\fq$ is the
largest $\rho$-invariant divisor of distinct smooth points which we can choose with this property.
We will show how the multiplier spectrum $\Sp(L^*,D)$ arises from the sections of $\pi^*L^*$
obtained by evaluating the global Baker function at different points of $X\setminus\{P,\rho P\}$.

To $X$ we can assign the subgroup
\[
\Gamma_X = \{\gamma\in\Gamma;\gamma\text{ extends hol.\ to }h:X\setminus\{P,\rho P\}\to\C^\times,\
\overline{\rho^*h} = h\}.
\]
Notice that $\Gamma_-\cap\Gamma_X=\R^+$ while $\Gamma_0\cap\Gamma_X=\{1\}$. 
Clearly we have an exact sequence
\[
1\to\Gamma_-\Gamma_0\Gamma_X\to\Gamma\stackrel{\ell}{\to} J_\R(X)\to 1.
\]
In particular this gives us homomorphisms
\[
\mu:\ker(\ell)\to\Gamma_0;\quad\chi:\ker(\ell)\to\Gamma_X
\]
defined by the unique factorisation
\[
\gamma = \gamma_-\mu(\gamma)^{-1}\chi(\gamma),\quad \gamma\in\ker(\ell),
\]
where $\gamma_-$ is normalised by $\gamma_-=1+O(\zeta^{-1})$.
\begin{lem}\label{lem:monodromy}
The global Baker function $\psi_W$ for this spectral data satisfies
\begin{equation}\label{eq:muchi}
\psi_W(\bt' +\bt) = \mu(\gamma)^{-1}\psi_W(\bt')\chi(\gamma),
\end{equation}
whenever $\gamma=\gamma(\bt)\in\Gamma_+\cap\ker(\ell)$. 
\end{lem}
\begin{proof}
The transformation due to $\Gamma_0$ comes directly from equation \eqref{eq:psit_0}. So let us
assume $\gamma$ has trivial $\Gamma_0$ factor and that $\ell(\gamma)\simeq\caO_X$. Since
we are dealing with a global Baker function there is a trivialisation $\varphi$ over $P$ which
extends holomorphically to $X$. Let $\varphi(\gamma)$ be the
trivialisation for $\caL\otimes\ell(\gamma)$ obtained by twisting the $1$-cocycle for
$(\caL,\varphi)$ by $\gamma$. Then
\[
\psi_W(\gamma) = (\frac{\sigma}{\varphi(\gamma)}-
j\frac{\bar\rho^*\sigma}{\varphi(\gamma)})\gamma,
\]
where $\sigma$ is the global section of $\caL(-\rho P)$ normalised by $\sigma|_P=\varphi|_P$. Since
$\gamma = \gamma_-\chi(\gamma)$ by assumption, the isomorphism $\caL\otimes\ell(\gamma)\simeq\caL$
equates $\varphi\gamma_-$ with $\varphi(\gamma)$, hence
\[
\psi_W(\gamma) = (\frac{\sigma}{\varphi} - j\frac{\bar\rho^*\sigma}{\varphi})\gamma_-^{-1}\gamma = 
\psi_W(1)\chi(\gamma).
\]
Equation \eqref{eq:muchi} follows easily from this by replacing $\caL$ with $\caL\otimes\ell(\gamma')$.
\end{proof}
Now let us restrict $\ell$ to $\Gamma_1\simeq\C$. Here it has kernel $\Lambda$ and we obtain 
homomorphisms
\[
\mu:\Lambda\to S^1;\quad\chi:\Lambda\to\Gamma_X.
\]
We may think of $\chi$ as a function on $\Lambda\times X\setminus\{P,\rho P\}$.
Thus to each point $Q\in X\setminus\{P,\rho P\}$ we get a function $\psi_W(z,Q):\C\to\H$ with the
properties
\[
\caD\psi_W(z,Q)=0,\quad \psi_W(z+\lambda,Q) = \mu(\lambda)^{-1}\psi_W(z,Q)\chi(\lambda,Q).
\]
Hence $\chi(\lambda,Q)\in \Sp(L^*,D)$.
\begin{thm}
Let $f:\C/\Lambda\to\HP^1$ be a non-constant conformal immersion of finite type from $Gr_\QKP$, 
with global Baker function
$\psi(z,Q)$ on $\C\times X\setminus\{P,\rho P\}$, and for which the flat bundle $L^*$ has monodromy
$\mu\in\Hom(\Lambda,S^1)$. Define
\[
\chi:\Lambda\times X\setminus\{P,\rho P\};\quad \chi(\lambda,Q) =
\psi(0,Q)^{-1}\mu(\lambda)\psi(\lambda,Q).
\]
For any pair of generators $\lambda_1,\lambda_2$ of $\Lambda$ the holomorphic map
\begin{equation}\label{eq:XtoSp}
X\setminus\{P,\rho P\}\to \Sp(L^*,D);\quad Q\mapsto (\chi(\lambda_1,Q),\chi(\lambda_2,Q))
\end{equation}
is surjective onto $\Sp(L^*,\caD)$. Moreover, this map factors through the covering map 
$X\setminus\{P,\rho P\}\to X_\fq\setminus\{P,\rho P\}$. 
\end{thm}
\begin{proof}
This map is clearly holomorphic and non-constant when $f$ is non-constant, and therefore the
image is an analytic subvariety of $\Sp(L^*,D)$. The image also contains annuli about each of
$P,\rho P$, by the symmetry of the real involution. Since $\Sp(L^*,D)$ has at most two
irreducible components, one about each of $P,\rho P$, the image must agree with $\Sp(L^*,D)$.
At each point of $Q\in \fq$ we know from  Theorem \ref{thm:fperiodic} 
that $\psi(z,Q)$ corresponds to a \qh section of $L^*$, i.e., it has trivial multiplier. Hence
the map \eqref{eq:XtoSp} factors through $X_\fq$.
\end{proof}

\subsection{A comparison of spectral curves.}
Taimanov \cite{Tai04} initially proposed that, when it has finite genus, the normalisation 
$\Sigma$ of $\Sp(L^*,D)$ should be the spectral curve, and this
is the definition used in \cite{BohLPP,BohPP}. But $\Sigma$ lacks
the subtlety necessary to be useful, because it throws away crucial information. It does this at
two levels: (i) it throws away the information contained in the divisor $\fq$, (ii) unless $X$ is
smooth, which it need not be, it throws away the information of singularities in $X$. Examples
of the latter case have been discussed by Taimanov himself, in the context of the generalised
Weierstrass representation, in \cite{Tai03,Tai05,Taisurv}.

The virtue of the $\Sp(L^*,D)$ is that it is directly constructed from the Dirac
operator, equally, the \qh structure $(L^*,D)$. This means that it is more properly
an invariant of the \qh curve $f:\C/\Lambda\to\HP^n$ given by \eqref{eq:Kodaira} 
(where $n+1=\dim_\H H^0_D(E)$). Provided we take $\fq$ to 
be the maximal divisor on which the multiplier of the Baker function is trivial, $X_\fq$ is 
likewise an invariant of this \qh curve. Congruence of $f$ in $\HP^n$, which is
the natural equivalence relation on such curves, gives a broader equivalence than
congruence of any of the projections of $f$ into $S^4$, but it is clear from all the discussions 
above that this is the correct notion of
equivalence from the point of view of spectral data. The virtue of $X_\fq$ over $\Sp(L^*,D)$ is
that it is algebraic and has no spurious singularities. 

For example, in the case of Example \ref{ex:HSL} earlier, $\fq$ is a divisor of distinct
points on $X\simeq\Ci$, and therefore $X_\fq$ is a partial resolution of
$\Sp(L^*,D)$: it keeps the essential information about where the multipliers are trivial but
discards the singularities caused by its immersion into the plane by \eqref{eq:XtoSp}.

This prompts the question of how to describe $X_\fq$ as a cover of $\Sp(L^*,D)$ without passing
through the QKP construction. I suspect the answer is something like the following: the kernel 
of the
Dirac operator should determine over $\Sp(L^*,D)$ a coherent analytic sheaf $\caE$ whose sections
represent functions $\psi(z,\zeta)$ which satisfy $\caD\psi=0$ for each $\zeta$, are holomorphic in
$\zeta$, and have the appropriate multiplier at $\zeta$. The sheaf of algebras
$\mathit{Hom}(\caE,\caE)$ would be the model for the structure sheaf of $X\setminus\{P,\rho P\}$. 
Notice that this is
consistent with the construction of $X$ as the compactification of $\Spec(\caA^\C_W)$ since
$\caA_W$ is isomorphic to an algebra of operators preserving the kernel of $\caD$.  
This instantly makes $\caE$ a maximal sheaf over $X\setminus\{P,\rho P\}$ which should be 
both rank one and torsion free.

\end{document}